\RequirePackage{fix-cm}
\documentclass[smallextended]{svjour3}       
\smartqed  
\usepackage{graphicx}
\usepackage[colorlinks,
linkcolor=blue,
anchorcolor=blue,
citecolor=blue]{hyperref}

\usepackage{cite}
 
\usepackage{mathtools}
\usepackage[figuresright]{rotating}
\usepackage{microtype}
\usepackage [latin1]{inputenc}
\usepackage[framemethod=tikz]{mdframed}
\usepackage{algorithm}
\usepackage{algorithmicx}  
\usepackage{algpseudocode}
\usepackage{float}
\usepackage{xcolor} 
\definecolor{myDeepPurple}{RGB}{80, 0, 80}
\definecolor{myDarkGreen}{RGB}{0, 80, 0}
\usepackage{color}
\usepackage{tabularx}
\usepackage{ragged2e}
\usepackage{booktabs}
\usepackage{amsmath,amssymb,amsfonts}
\usepackage{mathrsfs}
\usepackage{subfigure}
\usepackage{epstopdf}
\usepackage{comment}
\usepackage{enumerate}
\usepackage{setspace}
\newtheorem{assumption}{Assumption}

\def\F{\mathfrak{F}}

\def\R{{\mathbb{R}}}
\def\G{{\mathcal{G}}}

\def\argmin{\mathop{\rm arg\,min}}
\def\Argmin{\mathop{\rm Arg\,min}}

\def\YIR{{\bf IR}} 
\def\IR{{\bf GIR}} 

\def\SPA{{\bf SPA}}
\newtheorem{fact}{Fact}
\newcommand{\prox}{\operatorname{prox}}
\newcommand{\dom}{{\rm dom\,}}

\newcommand{\Diag}{{\rm Diag\,}}

\begin{document}

\title{A double iteratively reweighted algorithm for solving group sparse nonconvex optimization models
}

\titlerunning{Double Iteratively Reweighted Algorithm for Group Sparse Nonconvex Models}        

\author{Wanqin Nie \and Kai Tu \and Minglu Ye \and Shuqin Sun$^*$ 
}

\authorrunning{Wanqin Nie\and Kai Tu \and Minglu Ye\and Shuqin Sun} 

\institute{Wanqin Nie \at
          \email{wanqinnie@hotmail.com} 
          \\
            \emph{School of Mathematical Sciences, China West Normal University, Nanchong, Sichuan  637009, China}\\           
           \and
           Kai Tu \at
          \email{kaitu$\_$02@163.com} 
          \\
          \emph{School of Mathematical Sciences, Shenzhen University, Shenzhen 518061, China }\\
          \emph{The Guangdong Key Laboratory of Intelligent Information Processing, Shenzhen University, Shenzhen 518060, China} \\
          \emph{Shenzhen Key Laboratory of Advanced Machine Learning and Applications, School of Mathematical Sciences, Shenzhen University, Shenzhen 518000, China}
              \and
           Minglu Ye \at
           \email{yml2002cn@aliyun.com}
           \\
             \emph{Key Laboratory of Optimization Theory and Applications at China West Normal
University of Sichuan Province, School of Mathematical Sciences, China West Normal University, Nanchong, Sichuan  637009, China}
              \and
           Shuqin Sun \at
          \email{sunshuqinsusan@163.com} 
          \\
            \emph{School of Mathematical Sciences, China West Normal University, Nanchong, Sichuan  637009, China}\\ 
            \emph{Key Laboratory of Optimization Theory and Applications at China West Normal
University of Sichuan Province, Nanchong, Sichuan  637009, China}
}

\date{Received: date / Accepted: date}

\maketitle

\begin{abstract}
In this paper, we propose a double iteratively reweighted algorithm to solve nonconvex and nonsmooth optimization problems, where both the objectives and constraint functions are formulated by concave compositions to promote group-sparse structures.
At each iteration, we combine convex surrogate with first-order information to construct linearly constrained subproblems to handle the concavity of model.
The corresponding subproblems are then solved by  alternating direction method of multipliers to satisfy the specific stop criteria. 
In particular, under mild assumptions, we prove that our algorithm guarantees the feasibility of each subsequent iteration point, and the cluster point of the resulting feasible sequence is shown to be a stationary point. Additionally, we extend the group sparse optimization model, pioneer the application of the double iterative reweighted algorithm to solve constrained group sparse models (which exhibits superior efficiency), and incorporate a generalized Bregman distance to characterize the algorithm's termination conditions.
Preliminary numerical experiments show the efficiency
of the proposed method.
\end{abstract}

\section{Introduction}
Sparse optimization is a fundamental branch of mathematical optimization centered on seeking solutions with sparse structures (i.e., solutions containing a large proportion of zero or near-zero elements). In essence, sparse optimization techniques have seen remarkable expansion across diverse applications such as image processing \cite{EladAharon2006,SchusterTrettnerKobbelt2020}, signal processing \cite{TanEldarNehorai2014,TangNehorai2014}, machine learning \cite{AbernethyBachEvgeniou2009}, and recommendation systems \cite{KorenBellVolinsky2009},
to name a few.
 Among key sparse optimization models, compressed sensing \cite{CandesRombergTao2006} enables accurate reconstruction of sparse or compressible signals with far fewer measurements than traditional sampling; matrix completion \cite{KorenBellVolinsky2009,ZhangYangChenYang2018} recovers low-rank matrices from limited observed entries, widely used in recommendation systems to predict unobserved user preferences; and group sparse optimization \cite{PengChen2020,HuLiMengQinYang2017} extends sparsity to variable groups, facilitating structured high-dimensional data analysis in fields like bioinformatics and multimedia processing.

 Studying group sparse optimization is key as it addresses the structured, group-level sparsity of variables that traditional sparse optimization overlooks. Recently, the group sparse optimization model has gained wide research attention due to new application scenarios. It preserves anatomical structures in medical image denoising via group-level sparsity \cite{LiliPan2021,HuLiMengQinYang2017}, and ensures consistent variable selection across gene clusters in biological data regression \cite{SimonFriedman2013,QinHuYaoLeung2021}. Compared with sparse optimization problems, group sparse optimization models  have the following  advantages: (i) It faithfully models variable correlations, yielding superior interpretability because the inherent group structure is preserved, showing through the performance in image reconstruction\cite{LiliPan2021}. (ii) It demonstrates greater robustness against noise because group-level statistics are less susceptible to erratic changes in individual variables \cite{HuangZhang2010,ChenSelesnick2014}. (iii) It simplifies post-processing by reducing computational complexity \cite{ZhaWenYuan2021}.

 Based on group sparse optimization models, we first recall some unconstrained models. Its core in \cite{HuLiMengQinYang2017} is an unconstrained group sparse optimization model based on $\ell_{p,q}$ regularization, formulated as 
\begin{equation*}
    \min_{x \in \mathbb{R}^n} F(x)=\| A x-b\|^{2}+\lambda\| x\| _{p,q}^{q},
\end{equation*}
 where $x$ has a predefined group structure, $\| x\| _{p,q}$ is the $\ell_{p,q}$ norm inducing group sparsity with $p \geq 1$ and 0 $\leq q \leq1$, and $\lambda>0$ is the regularization parameter. To solve this model, the authors propose the proximal gradient method (PGM-GSO). The model is applied to simulated data (for group sparse signal recovery, with $\ell_{p,1/2}$ regularization showing superior performance) and real gene transcriptional regulation data. 
 Yuan and Lin \cite{YuanLin2006} proposed an unconstrained group sparse model that uses least squares as the loss function and the $\ell_{2,1}$-norm regularization term to impose group-level sparsity. 
Fan and Li \cite{FanLi2001} put forward the smoothly clipped absolute deviation (SCAD) regularization theory based on nonconcave penalized likelihood. This theory was later extended to the unconstrained group SCAD regularization model.

 Subsequently, we introduce several constrained group sparse optimization problems. For convex-constrained group sparse regularization problems with convex but non-smooth loss functions,\cite{ZhangPeng2022} have proposed constructing models by relaxing group sparse penalty terms using group capped-$\ell_1$, and designed the group smoothing proximal gradient (GSPG) algorithm for solution. To solve group sparsity-constrained minimization problems, \cite{JHCZ2024} have transformed them into equivalent weighted $\ell_{p,q}$ norm constrained models. Based on the properties of Lagrangian duality, homotopy technology is adopted to handle parameter adjustment, ensuring that the algorithm output is the stationary point of the original problem. In \cite{LiliPan2021}, the authors consider the following special group sparse problem:
\begin{equation}\label{proble-pan}
	\begin{array}{rl}
		\min &\sum\limits_{i=1}^m \psi(\|x_i\|),
        \quad 
	{\rm s.t.}  \quad  \|Ax-b\| \leq \sigma,~~Bx\leq h,
	\end{array}
\end{equation} 
where $x$ is partitioned into $m$ non-overlapping groups $\displaystyle\{x_i\}_{i = 1}^m$ with each $x_i \in \mathbb{R}^{n_i}$ and $\sum_{i = 1}^m n_i = n$,  the group-wise sparsity induces vanishing entries, $\psi$ is a capped folded concave function.
They propose smoothing penalty algorithm ($\SPA$ for short) to solve this problem and  show that the sequences generated by the algorithm converge to the directional stationary points. 
Based on the smoothing technique and penalty stragety, 
$\SPA$ needs to apply nonmonotone proximal gradient to solve a unconstrained nonconvex sub-problem until a specific criteria holds.


 In the research progress of the models, scholars have found that incorporating nonconvex relaxation terms into the objective function, such as the SCAD penalty \cite{JiangDF2012, FanLi2001, WangLChenG2007} and minimax concave penalty (MCP) \cite{ZhangC2010, HuangJBPMS2012} can better promote sparsity and reduce estimation bias. These properties make them particularly suitable for high-dimensional data analysis and variable selection tasks. In contrast, the use of nonconvex loss functions in the constraint component, including Tukey loss \cite{MostellerTukey1977}, Huber loss \cite{HuberPJ1992}, the 0-1 loss \cite{Schapire1990}, and its surrogate losses \cite{HastieTFJTR2001} enhances the robustness of the model, enabling effective handling of noise and outliers. Investigating problem models that integrate nonconvex relaxation functions with nonconvex loss functions holds significant importance in sparse optimization and complex data modeling. Despite the increased computational challenges associated with solving nonconvex problems, such models offer superior sparsity, robustness, and model flexibility. Consequently, the following nonconvex sparse optimization problem was considered in \cite{SSQTKP23}:
\begin{equation}\label{problem for ssq}
	\begin{array}{rl}
		\min\limits_{x\in \R^n} &\sum\limits_{i=1}^n \psi(|x_i|), \quad 
		{\rm s.t.} \quad  \sum\limits_{i=1}^m \phi\left(\left(a_i^Tx -b_i\right)^2\right)\leq \sigma,
	\end{array}
\end{equation}
which incorporates a nonconvex relaxation
in the objective function and a nonconvex loss function in the constraint part. A double iterative reweighting algorithm was devised by the authors to solve this problem, in which the  subproblem was approximated and solved using a convex subproblem solver.

Given the aforementioned necessity of employing group sparse optimization, we extend the problem \eqref{problem for ssq} to address the group sparisty of decision variables.
That is, we consider the following group sparse optimization problem:
\begin{equation}\label{Problem}
   \begin{array}{rl}
     \min\limits_{x\in \R^n} &\sum\limits_{i=1}^{q} \psi(\|x_{\mathcal{G}_i}\|),\quad
{\rm s.t.} \quad \sum\limits_{i=1}^{m} \phi \left(\left(a_{i}^T x-b_{i}\right)^2\right)\leq\sigma,
   \end{array}
\end{equation}
where $\left\{\mathcal{G}_1,\cdots,\mathcal{G}_q\right\}$ forms a partition of the index set $\left[ n \right] = \left\{ 1,2,\cdots,n \right\}$ and satisfies $\mathcal{G}_i \cap \mathcal{G}_j = \emptyset$ for all $ i \neq j $ and $\cup_{i=1}^{q}\mathcal{G}_i= \left[ n \right]$, $x_{\mathcal{G}_{i}}\in\R^{|\G_i|}$ denotes the $i$-th block of $x$ indexed by $\G_i$. Then we present $x_{\mathcal{G}}=\left(\|x_{\mathcal{G}_{1}} \|,\cdots , \|x_{\mathcal{G}_{q}}\|\right)^T$, which facilitates the subsequent discussion. Furthermore,  $m\ll n$, $a_{i}^T$ represents the $i$-th row of the matrix $A \in \R^{m \times n}$; $b \in \R^m$, $\sigma > 0$ and two functions $\psi, \phi:\R_{+} \rightarrow \R_{+}$  satisfy the following assumptions:

\begin{assumption}\label{assum}  {\rm(i)} The function $\psi$ is continuous and strictly concave over $\R_{+}$ with $\psi(0)=0$ and ${lim}_{t\to \infty} \psi\left(t\right)=\infty$. $\psi$ is differentiable on $(0,\infty)$, and $\psi'\left(t\right)>0$ for all $t \in \left(0,\infty \right)$. Furthermore, $\lim_{t\downarrow 0}\psi'(t)$ exists and satisfies $\lim_{t\downarrow 0}\psi'\left(t\right) \in \left(0,\infty \right)$.

{\rm(ii)} The function $\phi$ is continuous and concave over $\R_{+}$ with $\phi\left(0\right) = 0$. It is differentiable on $(0,\infty)$, and $\phi'\left(t\right) \geq 0$ for all $t>0$. Furthermore,  $\lim_{t\downarrow 0}\phi'\left(t\right)$ exists and satisfies $\lim_{t\downarrow 0}\phi'\left(t\right) \in \left( 0,\infty \right)$.

{\rm(iii)} The matrix $A$ is full row rank, and  $\sigma \in \left(0,\sum\limits_{i=1}^{m}\phi\left(b_{i}^2\right)\right)$.
\end{assumption}

 By making the above assumptions, we ensure that the feasible set of \eqref{Problem} does not contain $x = 0$. Moreover, since  $\psi$ and $\phi$ are differentiable and concave, their right derivatives $\psi'_+$ and $\phi'_+$ are continuous. This continuity proves useful in our subsequent analysis. It can also be observed that problem \eqref{Problem} reduces to problem \eqref{problem for ssq} in \cite{SSQTKP23} when $q=n$ given they both satisfy Assumption \ref{assum}.
However, due to the existence of gourp structure, the method in \cite{SSQTKP23} can't directly apply to solve the new model \eqref{Problem}.
The main contributions of this work are as follows:

\begin{itemize}
 \item[$\bullet$] {We extend the group sparse optimization model by introducing non-convex inequality constraints, enhancing its robustness in practical scenarios. When recovering group structured signals via the double iterative reweighted algorithm, problem \eqref{Problem} outperforms problem \eqref{problem for ssq} in recovery performance.}
    \item[$\bullet$]{The double iterative reweighted algorithm is first applied to solve constrained group sparse models, demonstrating superior efficiency compared to existing methods. When recovering the same group structured signal via the same model using different group sparse optimization algorithms, the algorithm proposed in this paper outperforms the one in \cite{LiliPan2021}. We note that while both this paper and \cite{SSQTKP23} use the iteratively reweighted algorithm, our model (a generalization of \cite{SSQTKP23} under group sparsity) lets our algorithm overcome the problem of solving the proximal operator in the inner loop. This challenge does not exist in \cite{SSQTKP23}.}
    \item[$\bullet$]{A generalized Bregman distance is incorporated to characterize termination conditions for nonlinear non-strongly convex subproblems, improving the applicability of the algorithm. Although  increase the difficulty of analyzing the convergence of the designed algorithm and verifying whether the subproblems meet the termination criteria, we have successfully overcome these challenges.}
\end{itemize}

The rest of this paper is organized as follows. In Section \ref{section2}, we introduce notations and preliminary materials to be used in subsequent analysis. In Section \ref{section3}, we present the algorithm flow of problem \eqref{Problem}, the construction process of the convex subproblem, and the progress of ADMM for solving the subproblem. Additionally, in this section, we also use Theorem \ref{Theoremforsubproblemkkt} to illustrate that the ADMM satisfies the three termination conditions set by the algorithm. In Section \ref{section4}, we establish the convergence analysis of the algorithm to validate its feasibility. Finally, in Section \ref{section5}, we verify the effectiveness of the algorithm through numerical comparison experiments conducted in MATLAB.

\section{Notation and preliminaries}\label{section2}

Throughout this paper, we denote $\R^n$ as the 
$n$-dimensional Euclidean space, $\R_{+}^n$ as the nonnegative orthant of $\R^n$ and let $\mathbb{N}$ denote the set of natural numbers. For a given vector $x\in \R^n$, $x_i$ denotes its $i$-th component. The notation $\sqrt{x}$ is used to represent the vector whose $i$-th component is $\sqrt{x_{i}}$, and $\|x\|$ denotes the $\ell_2$-norm of $x$. 
We let Diag($x$) be the $n\times n$ diagonal matrix with the $i$-th diagonal entry equal to $x_{i}$. For any two vectors $u,v\in\R^n$, $u\circ v$ and $\langle u ,v  \rangle$ denote their Hadamard (entry-wise) product, and   inner product, respectively. For any $x,y\in\R^n$,  they have the properties $\|x\|=\|x_{\G}\|$ and $\|x_{\mathcal{G}}-y_{\mathcal{G}}\| \leq \|(x- y)_\mathcal{G}\|$, which ensures that Proposition \ref{proposition}(ii) holds.

We say that an extended-real-valued function $f:\mathbb{R}^{n}\rightarrow\left(-\infty,+\infty\right]$ is proper if its domain ${\rm dom}\,f =\left\{x\in \mathbb{R}^n:\;f\left(x\right)<\infty\right\}$ is nonempty. A proper function is said to be closed if it is lower semicontinuous. For a proper function $f$, the regular subdifferential and Mordukhovich (limiting) subdifferential\cite[Definition 8.3]{variationalanalysis} of $f$ at $\bar{x}\in {\rm dom}\,f$ are defined respectively as
	\[
	\widehat{\partial}f(\bar{x})=\left\{v\in{\mathbb{R}^{n}}:\liminf_{x\rightarrow{\bar{x}},x\neq{\bar{x}}}\frac{f(x)-f(\bar{x})-v^T(x-\bar{x})}{\|x-\bar{x}\|}\geq{0}\right\},
	\]
	and
	\[
	\partial{f(\bar{x})}=\left\{v\in{\mathbb{R}^{n}}:\exists{x^{k}}\rightarrow{\bar{x}}~\mbox{and}~v^{k}\in\widehat{\partial}{f\left(x^{k}\right)}~\mbox{s.t.}~f(x^{k})\rightarrow{f(\bar{x})}~\mbox{and}~v^{k}\rightarrow{v}\right\}.
	\]
	By convention, we set $\widehat{\partial}f\left( x \right) = \partial{f\left( x \right)} = \emptyset$ if $x\notin{\rm{dom}}\, f$, and define the domain of $\partial{f}$ as ${\rm{dom}}\partial{f}=\left\{x\in \mathbb{R}^n:\;\partial{f\left(x\right)} \neq {\emptyset}\right\}$.
	When $f$ is proper and convex, the limiting subdifferential of $f$ at $x\in{\rm{dom}}\,f$ reduces to the classical subdifferential in convex analysis, i.e.,
	\[
	\partial{f} \left(x\right)=\left\{\xi\in \mathbb{R}^n:\;\xi^T \left(y - x\right) \leq {f\left(y\right) - f\left(x\right)}\ \ \forall y\in \mathbb{R}^n\right\};
	\]
	see \cite[Proposition 8.12]{variationalanalysis}.
	For a nonempty set $B$, the indicator function $\delta_{B}$ is defined as
	\[
	\delta_{B}(x)=
	\begin{cases}
		0,& {\rm if}\ x\in{B},\\
		\infty,& {\rm if}\ x\notin{B}.
	\end{cases}
	\]
	The normal cone (resp., regular normal cone) of $B$ at  $x \in {B}$ is defined as $N_{B}\left(x\right)=\partial{\delta_{B}} \left(x\right)$ (resp., $\widehat{N}_{B}\left(x\right)=\widehat{\partial}{\delta_{B}}(x)$), and the distance from any $x\in \mathbb{R}^n$ to $B$ is defined as ${\rm dist} \left(x,B\right) = \inf\limits_{y\in B}\|x - y\|$,  and if $B=\emptyset$, we  take $d\left(x, B\right)=\infty$ by convention. If $f$ is proper and closed, then the proximal mapping  of $f$ at $x\in \mathbb{R}^n$ with scaling parameter \( \lambda > 0 \) is defined as
	\begin{equation}
		\label{eq:def-prox}
		\begin{aligned}
			\prox_{\lambda f}(x)= \argmin_{u \in \mathbb{R}^n} \left\{f(u) + \frac{1}{2\lambda} \| x - u \|^2 \right\}.\\
		\end{aligned}
	\end{equation}

The Bregman distance \cite{BanerjeeMerugu2005,TuZhangGao2020}, also known as the Bregman divergence, has been extensively studied by scholars and is defined as follows.
\begin{definition}\label{def1}
     (Bregman distance) Let $f:\mathbb{R}^n\rightarrow \mathbb{R}\cup\{+\infty \}$ be a function that is continuously differentiable and strictly convex. For any two points $x\in \mathbb{R}^n$ and $x_0\in\mathbb{R}^n \cap \dom{f}$, the Bregman distance from $x_0$ to $x$ is defined as:
    \begin{equation*}
        D_{f}\left( x,x_0 \right)=f\left(x\right)-f\left(x_0\right)-\langle \nabla f\left(x_0\right), x-x_0 \rangle.
    \end{equation*}
\end{definition}

Using this definition, we define the Bregman distance from a point $x$ to a set $C\subset\mathbb{R}^n$ as $D_{f}\left(C,x\right)$ in the form:
    \begin{equation*}
        D_{f}\left(C,x\right) = \inf\limits_{y\in C} D_f\left(y,x\right).
    \end{equation*} 
 If $C=\emptyset$, we take $D_{f}\left(C,x\right)=\infty$ by convention.
%
%
If $f=\|\cdot\|^2$, then the Bregman distance reduces to the square of the Euclidean distance, that is, $D_{f}\left(C,x\right) = \inf\limits_{y\in C} \|y-x\|^2 = {\rm dist}^2(x,C)$.
To proceed with the convergence analysis in Section \ref{section4}, we apply \cite[Theorem 2.4]{SolodovSvaiter2000} with the choices $S=\R^n$ and $y^k=0$ for all $k$ and obtain the following Fact \ref{helptoxi0} concerning Bregman distance.

\begin{fact}
\label{helptoxi0}
 Let $f$ be defined in Definition \ref{def1}.  Given the sequence $\left\{x^n\right\}_{n\in\mathbb{N}}\subseteq \R^n$, if \[
   \lim_{n\to\infty}D^f \left(x^n,0\right) = 0,
    \]
    then we have that $\displaystyle \lim_{n\to \infty}x^n=0$.
\end{fact}

Consider problem \eqref{Problem} under Assumption \ref{assum}. We introduce the following notations for the convenience of the forthcoming discussion. For any $x\in\R^n$, $y\in\R^m_+$, and $z\in\R^q_+$, we define
\begin{align}\label{notation1}
\Phi(y)&=\sum\limits_{i=1}^m\phi(y_{i}),~\Phi_{+}'(y) = \left(\phi_{+}'(y_{1}),\phi_{+}'(y_{2}),\cdots,\phi_{+}'(y_{m})\right)^T,\notag\\
\Psi(z)&=\sum\limits_{i=1}^q\psi(z_{i}),~\Psi_{+}'(z) = \left(\psi_{+}'(z_{1}),\psi_{+}'(z_{2}),\cdots,\psi_{+}'(z_{q})\right)^T, \quad \text{and}\\
  G(x)&= x_{\mathcal{G}} = \left( \|x_{\mathcal{G}_1}\|,\|x_{\mathcal{G}_2}\|,\cdots,\|x_{\mathcal{G}_q}\| \right)^T,\notag
\end{align}
where $\G_i~(i = 1,2,\cdots,q)$ is defined in \eqref{Problem}.

We next present the definition of the stationary point of problem \eqref{Problem}.

\begin{definition} \label{KKT}
 A point $x\in \R^n$ is called a stationary point of problem \eqref{Problem} if there exists $\lambda \in \R_{+}$ such that the following conditions hold for $\left(x,\lambda\right)$:
\begin{align}
       &\displaystyle  \lambda\left(\Phi \left((Ax-b)\circ(Ax-b)\right)-\sigma\right) = 0,  \label{kkt1} \\
        &\displaystyle  \Phi \left((Ax-b)\circ(Ax-b)\right) \leq \sigma, \label{kkt2} \\
       &\displaystyle 0 \in \Psi_{+}' \left(x_{\mathcal{G}}\right) \circ \partial G(x) +2\lambda \sum\limits_{i=1}^{m} \phi_{+}'\left( \left(a_i^T x-b_i \right)^2\right) \left(a_i^T x -b_i\right)a_i,  \label{kkt3}
\end{align}
\end{definition}

 The following fact reveals 
that any local minimzer of problem \eqref{Problem} is the stationary point under Mangasarian-Fromovitz constraint qualification ~\eqref{def:MFCQ}. This follows by sequentially applying arguments analogous to those in \cite[Proposition 2.1]{SSQTKP23} and then \cite[Proposition 2.2]{SSQTKP23} under Assumption \ref{assum}, exploiting the local Lipschitz continuity of $\psi\circ \|\cdot\|$.

\begin{fact}\label{implyMFCQ}
  Suppose that Assumption \ref{assum} holds.
       If there exists $y\in \R^m$ such that $y = Ax-b$ and $\sum\limits_{i=1}^m \phi \left(y_i^2\right) = \sigma$,
       where
$      \sigma \notin \left\{k \bar{\phi}: \, k = 1,\cdots,m\right\}$, with
    $\bar{\phi}=\sup\limits_{t\in \R_+}\phi (t) \in \left(0,\infty\right]$.
       Then Mangasarian-Fromovitz constraint qualification holds, that is
       \begin{equation}\label{def:MFCQ}
    \begin{aligned}
        \Phi \left((Ax-b)\circ (Ax-b)\right) = \sigma \Longrightarrow \sum\limits_{i=1}^{m}\phi_{+}'\left((a_i^Tx -b_i)^2\right) \left(a_i^T -b_i\right) a_i \neq 0.
    \end{aligned}
\end{equation}
Furthermore,  
any local minimizer   of  problem \eqref{Problem} is its stationary point.
\end{fact}

\section{Double iteratively reweighted algorithm for group sparsity models}\label{section3}

In this section, we present our algorithm to solve problem \eqref{Problem} whose feasible set is defined by 
 $\F=\left\{x\in \R^n:\Phi \left((Ax-b) \circ (Ax-b)\right) \leq \sigma \right\}$.  
We begin by outlining the idea of algorithm design based on the assumptions in Fact \ref{implyMFCQ}. Taking advantage of the concavity of $\phi$ and $\psi$, 
given $x^k$, we have that
\begin{align*} 
  &\sum\limits_{i=1}^{q} \psi \left(\|x_{\mathcal{G}_i}\| \right) \leq \sum\limits_{i=1}^{q} \psi \left(\|x_{\mathcal{G}_i}^k\| \right) +
  \sum\limits_{i=1}^{q} \psi_{+}' \left(\|x_{\mathcal{G}_i}^k\| \right)
   \left(\|x_{\mathcal{G}_i}\|-\|x_{\mathcal{G}_i}^k\| \right), \quad \text{and}\\
     &\sum\limits_{i=1}^m\phi \left(\left(a_{i}^Tx-b_{i}\right)^2\right)
  \leq\Phi \left((Ax^k-b)\circ(Ax^k-b)\right)\\
&~~~~~~~~~~~~~~~~~~~~~~~~~~~~~~+\sum\limits_{i=1}^m\phi_{+}' \left( \left(a_{i}^Tx^k-b_{i}\right)^2\right) \left[\left(a_{i}^Tx-b_{i}\right)^2 - \left(a_{i}^Tx^k-b_{i}\right)^2\right].
\end{align*}
The above two inequalities inspire   us to formulate the following relaxed subproblem:
\begin{equation}\label{original_subproblem}
 \begin{aligned}
  \min\limits_{x\in \R^n}&\sum\limits_{i=1}^{q}\psi_{+}' \left(\|x_{\mathcal{G}_i}^k\|\right) \cdot \|x_{\mathcal{G}_i}\|\\
  \text{s.t.}& ~\Phi \left((Ax^k-b)\circ(Ax^k-b)\right)\\
  ~&+\sum\limits_{i=1}^m\phi_{+}' \left( \left(a_{i}^Tx^k-b_{i}\right)^2\right)
  \left[\left(a_{i}^Tx-b_{i}\right)^2-\left(a_{i}^Tx^k-b_{i}\right)^2\right] \leq \sigma,
\end{aligned}
\end{equation}
 which, by using the notations in \eqref{notation1}, can be reformulated  as
 \begin{equation}
 \label{zproblem}
  \min\limits_{x\in \R^n} \|\omega_{\mathcal{G}}^k\circ G(x)\|_{1}, \quad 
  \text{s.t.} \quad A_{k}x-u =b^k,\quad  u\in U^k,
\end{equation}
where 
\begin{equation}\label{notation2}
 \begin{aligned}
&\displaystyle  \omega_{\mathcal{G}}^k \!= \! \Psi_{+}' \!\!\left(x_{\mathcal{G}}^k\right),~A_{k}\!=\!\Diag \!\!\left(\upsilon^k\right)\! A~{\rm with}~\upsilon^k \!\!=\!\sqrt{\Phi_{+}'\!\! \left(y^k\circ y^k\right)}~{\rm and}~y^k \!\!=\! Ax^k\!\!-\!b,\\
&\displaystyle  b^k =\upsilon^k \circ b,
  ~~\sigma_{k}\!=\! \sigma \!+\! \|A_{k}x^k \!-\! b^k\|^2 \!-\! \Phi \left(y^k\circ y^k\right), ~~U^k= \{ u: \| u\|^2 \!\leq \! \sigma_{k} \}.
 \end{aligned}
\end{equation}

Let $\mathfrak{F}_k$ denote the feasible set of \eqref{zproblem}, that is, $\mathfrak{F}_k = \{ x\in \mathbb{R}^n : \|A_k x - b^k\|^2 \leq \sigma_k \}$.
Suppose further that $x^k$ is a feasible point of problem \eqref{Problem}. Then from \eqref{original_subproblem} and  \eqref{notation2}, we can see that $x^k\in\mathfrak{F}_k$. This ensures nonemptiness of $\mathfrak{F}_k$. Moreover, since the objective in \eqref{zproblem} is level-bounded which is implied by the positivity of $\omega_{\mathcal{G}}^k$, the solution set of subproblem \eqref{zproblem} is nonempty.
Therefore, the subproblem \eqref{zproblem} is well defined if $x^k\in\mathfrak{F}$.

Note that the functions in subproblem \eqref{zproblem} are convex. Thus, if subproblem \eqref{zproblem} is valid,  it  can be inexactly solved by some first order algorithms efficiently, such as alternating direction method of multipliers (ADMM) that we can see later. 
Here we adopt the Bregman distance to measure one of the termination conditions of the subproblem computation instead of the usual European distance used in \cite{SSQTKP23}. 
However, the resulting approximate point maybe infeasible to problem \eqref{zproblem}. 
In order to guarantee the feasibility, we also need to retract the approximate solution  to $\F_k$. 

Based on above reason, we subsequently present our Algorithm \ref{algorithm1} ($\IR$ for short) below for solving problem \eqref{Problem} under the assumptions in Fact \ref{implyMFCQ}.
    \begin{algorithm}[H]
        \caption{Doubly iteratively reweighted algorithm.
        }\label{algorithm1}
        \begin{algorithmic}[1]
             \State \textbf{Input:} Take a positive diminishing sequence $\{\tau_{k}\}_{k\in \mathbb{N}}$, and a summable positive sequence $\{\mu_{k}\}_{k\in \mathbb{N}}$. Set  $x^0=A^{\dagger}b$ and   $k=0$.  
             \While{the stopping criteria does not hold}
             \State Compute $\omega_{\mathcal{G}}^k$, $y^k$,$\upsilon^k$, $A_{k}$, $b^k$ and $\sigma_{k}$ as defined in \eqref{notation2}.  
             \State Find a pair $\left(\widetilde{x}^{k+1},\widetilde{u}^{k+1}\right)$ by inexactly solving subproblem \eqref{zproblem} such that the following three conditions hold:
             \begin{align}
             & \displaystyle D^f \left(\partial \left(\omega_{\mathcal{G}}^k\circ G\right) (\tilde{x}^{k+1}) + A_{k}^T N_{ U^k}\left(\tilde{u}^{k+1}\right),\; 0 \right) \leq\epsilon_{k}, \label{zzzz1} \\
             &\displaystyle  \|A_{k}\tilde{x}^{k+1} - b^k - \tilde{u}^{k+1}\| \leq \epsilon_{k},\label{zzzz2}\\
             &\displaystyle \|\omega_{\mathcal{G}}^k\circ G( P_{k} ( \tilde{x}^{k+1})  )\|_{1} \leq \|\omega_{\mathcal{G}}^k\circ x_{\mathcal{G}}^k\|_{1} + \mu_{k}, \label{zzzz3}
             \end{align} 
             where $\epsilon_{k} \in \left(0,\min \left\{\sigma_{k},\sqrt{\sigma_{k}},\tau_{k} \right\}\right]$,  
\begin{align*}
    P_{k}(x)&=\begin{cases}
                    x, & {\rm if}\ \|A_kx - b^k\|^2 \leq \sigma_k,\\
                    \left(1-\frac{\sqrt{\sigma_{k}}}{\|A_{k}x-b^{k}\|}\right)A^\dagger b + \frac{\sqrt{\sigma_{k}}}{\|A_{k}x-b^{k}\|}x, & {\rm otherwise}.
                \end{cases}
    \end{align*} 
            \State Set $x^{k+1} = P_{k}\left(\tilde{x}^{k+1}\right)$, and   $k=k+1$.    

        
    \EndWhile
    
        \end{algorithmic}
    \end{algorithm}


\begin{remark}\label{feasiblility}
We have the following comments  for Algorithm \ref{algorithm1}.
\begin{itemize}
    \item[(i)] Due to the definition of $\mathfrak{F}$ and $A^{\dagger}$, we have that $x^0=A^{\dagger}b \in \mathfrak{F}$.
    \item[(ii)] From \cite{LSM2020}, we have
\begin{equation}\label{subd_norm2}
\partial \left(\omega^k_\mathcal{G} \circ G\right)\left({x}^{k+1}\right) = \omega^k_\mathcal{G} \circ \partial G\left({x}^{k+1}\right)   =J_{1}\times J_{2}\times \cdots \times J_{q}.
\end{equation}
where for any $i\in \{1,\cdots, q\}$, 
\begin{equation*}
    J_i=
    \begin{cases}
\left(\omega_{\mathcal{G}}^k\right)_{i} \cdot \frac{{x}_{{\mathcal{G}}_i}^{k+1}}{\|{x}_{{\mathcal{G}}_i}^{k+1}\|},& {\rm if}\  {x}_{{\mathcal{G}}_i}^{k+1}\neq 0,\\
    B\left(0,(\omega_{\mathcal{G}}^k\right)_{i}), & {\rm otherwise}.
    \end{cases}
\end{equation*}

\item[(iii)] In view of the concavity of $\phi$, we have  $P_{k}(x) \in \mathfrak{F}_k \subseteq \mathfrak{F}$ for all $x \in \R^n$.

\item [(iv)]     It follows from \cite[Lemma 3.1]{SSQTKP23} that for each $k$, $0<\sigma_k\leq\sigma$.
\end{itemize}

\end{remark}


\subsection{Subproblem computation}

In this subsection, we apply proximal ADMM to solve the inner subproblem \eqref{zproblem}.
As outlined in Algorithm \ref{algorithm1}, if $x^k\in\F$, then subproblem \eqref{zproblem} is well defined by virtue of having a nonempty solution set. Thus in view of Remark \ref{feasiblility}, which shows that $x^0\in\F$ and all generated iterates $x^{k+1}$ remain in $\F$, we can see that Algorithm \ref{algorithm1} is well-defined by showing that the tuple $\left(\widetilde{x}^{k+1},\widetilde{u}^{k+1}\right)$ generated by the subproblem solver satisfies the inexact criteria \eqref{zzzz1}, \eqref{zzzz2} and \eqref{zzzz3}.



Note that subproblem \eqref{zproblem} is a convex problem with two block variables, decision varibale $x$  and the slack variable $u$, and  one linear constraint.
These structures motivate us apply  proximal ADMM to solve it.
Let
$L_\beta$ be the augmented lagrangian function of \eqref{zproblem}, that is,
\begin{align*}
      L_\beta (x,u;z)=\|\omega^{k}_{\mathcal{G}} \circ G(x)\|_1  +\delta_{U^k} (u) -{z}^T(A_k x-u - b^k ) +\frac{\beta}{2} \| A_k x-u-b^k\|^2,
\end{align*}
where  $z$ is the Lagrange multiplier,  and
 $\beta$ denotes the penalty coefficient.  
 Then the basic process of solving \eqref{zproblem} in Algorithm \ref{algorithm1} via the proximal ADMM is shown in Algorithm~\ref{Algorithm2}.

\begin{algorithm}[H]
    \caption{Proximal ADMM for \eqref{zproblem}.}\label{Algorithm2}  
    \begin{algorithmic}[1]
             \State \textbf{Input:} Take $\beta>0$, $r\in \left(0,\frac{1+\sqrt{5}}{2}\right)$, $\rho>0$, and   $\left(x^{k,0}, u^{k,0}, z^{k,0}\right)$. Set $l=0$.               
            \While{the point $\left(x^{k,l},u^{k,l}\right)$ does not satisfy \eqref{zzzz1}, \eqref{zzzz2} and \eqref{zzzz3} }
            \State Compute
            \[
            x^{k,l+1}\! \in \!\Argmin\limits_{x\in \R^n} \left\{ \!L_{\beta} \left(x,u^{k,l};z^{k,l} \right) \!+\! \frac{1}{2}\left(x-x^{k,l}\right)^T \! \left(\rho I \!-\!\beta A_{k}^T A_k\right) \!\left(x \!-\!x^{k,l}\right)\!\right\}.
            \]
          
            \State Update $ u^{k,l+1}= \argmin\limits_{u\in \R^m} \left\{L_{\beta}(x^{k,l+1},u;z^{k,l})\right\}$, that is,  
            \[
             u^{k,l+1} =  \frac{\sqrt{\sigma_k}\left(A_kx^{k,l+1}-b^k- \frac{1}{\beta} z^{k,l} \right) }{\max\left\{ \sqrt{\sigma_k}, \|A_kx^{k,l+1}-b^k- \frac{ 1}{\beta} z^{k,l} \| \right\} } .
            \]

            \State Calculate $z^{k,l+1}=z^{k,l}- r\beta \left(A_kx^{k,l+1}-b^k-u^{k,l+1}\right)$ and set  $l=l+1$.
        \EndWhile  
    \end{algorithmic}
\end{algorithm}

\begin{remark}
\label{remarkforbanzhengding}
    We point out that in Algorithm \ref{Algorithm2}, $\rho=\bar{L}\beta >0$ where
    \[
    \bar{L}=\max\limits_{i} \left\{\phi_{+}' \left(\left(a_i^T x^k -b_i\right)^2\right)\right\}\lambda_{\max} \left(A^TA\right).
    \]
    Since $A$ has full row rank,  it follows that $\lambda_{\max} \left(A^TA\right)>0$. Then we have $\bar{L}\geq \lambda_{\max}(A_k^TA_k)$ and $\rho I -\beta A_k^TA_k \succeq 0$.
\end{remark}

\begin{remark}
    We present the details on the computation of $x^{k,l+1}$. 
Using the definition of $L_\beta$ and  \eqref{notation2}, the detailed update rule for $x^{k,l+1}$ is
\begin{equation*}
        x^{k,l+1} \in \Argmin\limits_{x\in \R^n}\sum\limits_{i=1}^{q}\left\{ \frac{({w}^{k}_\mathcal{G})_{i}}{\rho} \cdot \|x_{\mathcal{G}_i}\|+ \frac{1}{2}\|x_{\mathcal{G}_i}-v^{k,l}_{\mathcal{G}_i}\|^{2} \right\},
\end{equation*}
where ${v}^{k,l} =x^{k,l}-\frac{\beta}{\rho}A_k^{T} \left(A_kx^{k,l}-b^k-u^{k,l}-\frac{z^{k,l}}{\beta}\right)$. Using the proximal mapping of $\|\cdot\|$, we obtain that  $x^{k,l+1}=\left(\left(x_{\mathcal{G}_1}^{k,l+1}\right)^{T},\cdots,\left(x_{\mathcal{G}_q}^{k,l+1}\right)^{T}\right)^{T},$ where for $i=1,2,\cdots,q$,
            \begin{equation*}
	x_{\mathcal{G}_i}^{k,l+1}=\begin{cases}
			\left(1-\frac{\frac{({w}^k_{\mathcal{G}})_{i}}{\rho}}{\|v_{\mathcal{G}_i}^{k,l}\|}\right)\cdot v_{\mathcal{G}_i}^{k,l}, & {\rm if}\ \|x_{\mathcal{G}_i}^{k,l}\| \geq \frac{({w}^k_{\mathcal{G}})_{i}}{\rho},\\
			0 ,& {\rm otherwise}.
		\end{cases}
	\end{equation*}
\end{remark}

\begin{theorem}\label{Theoremforsubproblemkkt}
    Suppose that Assumption \ref{assum} hold, and that the sequences $\{x^{k,l}\}_{l\in\mathbb{N}}$ and   $\{u^{k,l}\}_{l\in\mathbb{N}}$ are generated by Algorithm \ref{Algorithm2}.
    Then for sufficiently large $l$, the  pair $\left(x^{k,l},u^{k,l}\right)$ satisfies conditions \eqref{zzzz1}-\eqref{zzzz3}. 
\end{theorem}

\begin{proof}
  Before we proceed with the discussion, we first point out from \cite[Theorem B.1]{Fazel2013} that the sequence $\{x^{k,l}\}_{l\in\mathbb{N}}$ converges to an optimal solution $x^{k,\star}$ of subproblem~\eqref{zproblem}, and letting $l\to \infty$, we also have
\begin{align}\label{adjacent changes}
       & \|x^{k,l} -x^{k,l+1}\| \to 0,~\|z^{k,l} -z^{k,l+1}\| \to 0,~\|u^{k,l} -u^{k,l+1}\| \to 0.
\end{align}
Furthermore, by the definitions of $\omega_{\mathcal{G}}$ and $G(\cdot)$, we get
    \begin{align}\label{objective limit}
      \|\omega^{k}_{\mathcal{G}} \circ G(x^{k,l})\|_1 \to \|\omega^{k}_{\mathcal{G}} \circ G(x^{k,\star})\|_1.
    \end{align}
In fact, using the optimality conditions during the update of $x$ and $u$ in steps 3-4 of Algorithm \ref{Algorithm2}, we derive that
\begin{align}
\label{relation for 1314}
       \begin{cases}
       \displaystyle 0 \in  \omega^{k}_{\mathcal{G}} \circ \partial G (x^{k,l+1})-A_{k}^T z^{k,l} +\beta A_{k}^T \left(A_kx^{k,l+1} - b^k -u^{k,l}\right) \\
       ~~~~~+ \left(\rho I-\beta A_{k}^T A_k\right) \left(x^{k,l+1}-x^{k,l}\right),\\
      \displaystyle  0 \in N_{U^k} (u^{k, l+1}) +z^{k,l} -\beta \left(A_k x^{k,l+1} -b^k -u^{k, l+1}\right),\\
       \displaystyle  z^{k, l+1} =z^{k,l}- r\beta \left(A_kx^{k,l+1} -b^k -u^{k, l+1}\right),
       \end{cases}
\end{align}
the first two relations in \eqref{relation for 1314} implies that
\[
-\!\beta A_{k}^T\!\!\left(\!u^{k, l+1}\! -\!u^{k,l}\!\right)\!-\!\left(\!\rho I\!-\!\beta A_{k}^T A_k\!\right)\left(\!x^{k,l+1}\!-\!\!x^{k,l}\!\right)\!\in\! A_{k}^T N_{U^k} (\!u^{k, l+1}\!) \!+\!\omega^{k}_{\mathcal{G}}\!\circ \! \, \partial G (\!x^{k,l+1}\!).
\]
Combing this relation with \eqref{adjacent changes}, Remark \ref{remarkforbanzhengding} and the definition  $\epsilon_k$, for all sufficiently large $l$, we have
\begin{align*}
     & \|-\beta A_{k}^T \left(u^{k, l+1} -u^{k,l}\right) - \left(\rho I-\beta A_{k}^T A_k\right) \left(x^{k,l+1} -x^{k,l}\right)\| \\
    & \leq \|-\beta A_{k}^T\| \|u^{k, l+1} -u^{k,l}\| + \|\rho I-\beta A_{k}^T A_k\| \|x^{k,l+1} -x^{k,l}\| \leq \epsilon_k.
\end{align*}
Using the definition of $D_f$, this further implies that for all sufficiently large $l$,
\begin{align*}
       &D_f\left( \omega^{k}_{\mathcal{G}} \circ \partial   G   (x^{k,l+1}) +A_{k}^T N_{ U^k }(u^{k,l+1}),\; 0\right)\\
       &\leq D_f \left(-\beta A_{k}^T \left(u^{k, l+1} -u^{k,l}\right) -\left(\rho I-\beta A_{k}^T A_k\right) \left(x^{k,l+1} -x^{k,l}\right),\; 0 \right) \leq \epsilon_k,
\end{align*}
which confirms that criterion \eqref{zzzz1} holds. For criteria \eqref{zzzz2}, from the second relation in \eqref{adjacent changes} and the identity $\| A_kx^{k,l+1} -b^k -u^{k, l+1}\|=\frac{1}{r \beta} \| z^{k,l} - z^{k,l+1}\|$ (from the third relation in \eqref{relation for 1314}), we have that for all sufficiently large $l$, $\| A_kx^{k,l+1} -b^k -u^{k, l+1}\| \leq \epsilon_k$ holds. This means \eqref{zzzz2} holds.
   
It remains to verify criterion \eqref{zzzz3}. By use of \eqref{adjacent changes}, \eqref{objective limit} and the definition of $P_k$ in Algorithm \ref{algorithm1}, we have that $\lim\limits_{l\to \infty} \|\omega^{k}_{\mathcal{G}} \circ \left(P_k(x^{k,l+1})\right)_\mathcal{G} \|_1= \lim\limits_{l\to \infty} \|\omega^{k}_{\mathcal{G}} \circ x^{k,l+1}_\mathcal{G} \|_1 = \lim\limits_{l\to \infty} \|\omega^{k}_{\mathcal{G}} \circ x^{k,l}_\mathcal{G} \|_1 = \|\omega^{k}_{\mathcal{G}} \circ x^{k,\star}_\mathcal{G}\|_1$. This together with $\mu_k >0$ implies that for sufficiently large $l$,
\[
\|\omega^{k}_{\mathcal{G}} \circ \left(P_k(x^{k,l+1})\right)_\mathcal{G} \|_1 < \|\omega^{k}_{\mathcal{G}} \circ x^{k,\star}_\mathcal{G}\|_1 + \mu_k.
\]
Note that $x^k$ is feasible for the subproblem \eqref{zproblem} and $x^{k,\star}$ is an optimal solution of \eqref{zproblem}, so $\|\omega^{k}_{\mathcal{G}} \circ x^{k,\star}_\mathcal{G}\|_1 \leq \|\omega^{k}_{\mathcal{G}} \circ {x}^k_\mathcal{G}\|_1 $. Thus, we obtain that
\begin{align*}
	\|\omega^{k}_{\mathcal{G}} \circ \left(P_k(x^{k,l+1})\right)_\mathcal{G} \|_1 + \mu_k < \|\omega^{k}_{\mathcal{G}} \circ {x}^k_\mathcal{G}\|_1 +\mu_k,
\end{align*}
 which yields that criteria \eqref{zzzz3} holds.
 \end{proof}
 
\section{Convergence analysis}\label{section4}
We  establish the following proposition through a similar proof as demonstrated in \cite[Proposition 3.1]{SSQTKP23}. Here we omit the proof for brevity.
\begin{proposition}\label{proposition}
Suppose Assumptions \ref{assum} and the conditions in Fact \ref{implyMFCQ} hold.
    Consider Algorithm \ref{algorithm1}  to solve problem \eqref{Problem}. Let $\{x^k\}_{k\in\mathbb{N}}$ and $\{\tilde{x}^k\}_{k\in\mathbb{N}}$,    be generated by Algorithm \ref{algorithm1}. Then  the following statements hold:
\begin{enumerate}[\rm (i)]
\item The sequences $\{x^k\}_{k\in\mathbb{N}}$, $\{\tilde{x}^k\}_{k\in\mathbb{N}}$, $\{x_{\mathcal{G}}^k\}_{k\in\mathbb{N}}$, $\{\tilde{x}_{\mathcal{G}}^k\}_{k\in\mathbb{N}}$ are bounded, and for all $k$, the following inequality holds
\begin{equation*}
    \Psi\left(x_{\mathcal{G}}^{k+1}\right)-\Psi \left(x_{\mathcal{G}}^k \right) \leq \mu_k.
\end{equation*}
 Furthermore, the sequence $\left\{\Psi \left(x_{\mathcal{G}}^{k}\right)\right\}_{k\in\mathbb{N}}$ is convergent.
\item There exists constant $M>0$ such that for all $k$,
\begin{align*}
    \|x_{\mathcal{G}}^{k+1}-\tilde{x}_{\mathcal{G}}^{k+1}\|
     \leq \| \left(x^{k+1}-\tilde{x}^{k+1} \right)_\mathcal{G}\|
     = \| x^{k+1}-\tilde{x}^{k+1} \|
    &\leq \sqrt{\epsilon_k}M.
\end{align*}
\item  $\lim\limits_{k\to \infty}\|\tilde{x}_{\mathcal{G}}^{k+1}-x_{\mathcal{G}}^k\|=0$.
\end{enumerate}
\end{proposition}

\begin{remark}\label{boundedness}
     By Proposition \ref{proposition} (i), the sequences $\{\omega_{\mathcal{G}}^k\}_{k\in\mathbb{N}}$, $\{A_k\}_{k\in\mathbb{N}}$, $\{b^k\}_{k\in\mathbb{N}}$ and $\{\sigma_k\}_{k\in\mathbb{N}}$ generated by Algorithm \ref{algorithm1}  are all bounded, and hence the sequence $\{A_k \tilde{x}^{k+1}\}_{k\in\mathbb{N}}$ is also bounded.
\end{remark}

To show the convergence analysis of Algorithm \ref{algorithm1}, we first note from \eqref{zzzz1} that for each $k$, there exists  $\xi^k \in \R^n$ with $D^f \left(\xi^k,0\right)\leq \epsilon_k$ such that
\begin{equation*}
   \xi^k \in \omega_{\mathcal{G}}^k \circ \partial G \left(\tilde{x}^{k+1}\right) + A_{k}^T N_{U^k} \left(\tilde{u}^{k+1})\right),
\end{equation*}
which, combined with $\lim\limits_{k\rightarrow \infty}\epsilon_k= 0$, implies $\lim\limits_{k\rightarrow \infty} D^f \left(\xi^k,0\right)= 0$. This together with Fact \ref{helptoxi0} yields
\begin{equation}\label{xi2}
\lim_{k\to\infty}\|\xi^k\|=0.
\end{equation}
Using the definition of normal cone and the fact that $\sigma_k>0$, we further derive the existence of $\tilde{\lambda}_k \geq 0$ satisfying
\begin{align}\label{normal cone}
      \xi^k \in \omega_{\mathcal{G}}^k \circ \partial G (\tilde{x}^{k+1}) +\tilde{\lambda}_k A_k^T  \tilde{u}^{k+1}\quad {\rm and} \quad \tilde{\lambda}_k \left(\|\tilde{u}^{k+1}\|^2 -\sigma_k\right) = 0.
 \end{align}
Define 
\begin{equation}\label{def_v^k}
\tilde{v}^{k+1}=A_k \tilde{x}^{k+1} - b^k -\tilde{u}^{k+1}.
\end{equation}
By \eqref{zzzz2}, we have $\|\tilde{v}^{k+1}\| \leq \epsilon_k \leq \min \left\{ \sigma_k,\sqrt{\sigma_k}\right\}$. Furthermore, from \eqref{normal cone} it follows that
\begin{align}\label{argmin}
      \tilde{x}^{k+1} &\in \Argmin\limits_{x\in \R^n} \left\{(\omega_{\mathcal{G}}^k)^T G(x) - (\xi^k)^Tx + \frac{1}{2} \tilde{\lambda}_k \left(\|A_k x-b^k-\tilde{v}^{k+1}\|^2\right)\right\},
 \end{align}
and
\begin{align}\label{bianjie}
     \tilde{\lambda}_k \left(\|A_k \tilde{x}^{k+1}-b^k-\tilde{v}^{k+1}\|^2-\sigma_k\right) = 0.
 \end{align}

We are now ready to establish the subsequential convergence of the iterates generated by Algorithm  \ref{algorithm1}, following an argument analogous to that in \cite[Theorem 3.1]{SSQTKP23}. For completeness and self-containment, we present the full proof below.
\begin{theorem}\label{theorem}
    Consider \eqref{Problem} under Assumptions \ref{assum}, Facts \ref{helptoxi0} and \ref{implyMFCQ}.  Let $\{x^k\}_{k\in\mathbb{N}}$, $\{\tilde{x}^k\}_{k\in\mathbb{N}}$, $\{x_{\mathcal{G}}^k\}_{k\in\mathbb{N}}$ and $\{\tilde{x}_{\mathcal{G}}^k\}_{k\in\mathbb{N}}$ be generated by Algorithm \ref{algorithm1}, and let $\{\tilde{\lambda}_k\}_{k\in\mathbb{N}}$ be defined in \eqref{normal cone}. Then the following statements hold:
\begin{enumerate}[\rm (i)]
\item $\mathop{\lim\inf}\limits_{k \to \infty}\tilde{\lambda}_k >0$.
\item $\lim\limits_{k \to \infty}\phi_{+}'\left( \left(a_i^T x^k -b_i \right)^2 \right) \left(a_i^T \left(\tilde{x}^{k+1}-x^k\right) \right) = 0$ for all $i$.
\item $\Phi \left((Ax^{*}-b)\circ (Ax^{*}-b) \right) = \sigma$ for every accumulation point $x^{*}$ of $\{x^k\}_{k\in\mathbb{N}}$.
\item The sequence $\left\{\tilde{\lambda}_k\right\}_{k\in\mathbb{N}}$ is bounded.
\item Every accumulation point of $\left\{x^k\right\}_{k\in\mathbb{N}}$ is a stationary point of \eqref{Problem}.
\end{enumerate}
\end{theorem}

\begin{proof} We compelte the proof one by one.

  (i): If, on the contrary, $\mathop{\lim\inf}\limits_{k \to \infty}\tilde{\lambda}_k =0$, then there exists a subsequence $\{k_t\}_{k_t\in\mathbb{N}}$ of $\{k\}_{k\in\mathbb{N}}$ such that $\lim\limits_{t\to \infty}\tilde{\lambda}_{k_t} = 0$. Furthermore, by Proposition \ref{proposition}(i) and (iii), we also have $\lim\limits_{t\to \infty}x_{\mathcal{G}}^{k_t} = x_{\mathcal{G}}^{*}$ and $\lim\limits_{t\to \infty}\tilde{x}_{\mathcal{G}}^{k_t +1} = x^{*}_{\mathcal{G}}$. Let $\{\xi^{k_t}\}_{k_t\in\mathbb{N}}$ satisfy
\begin{align}\label{xi1}
     \xi^{k_t} &\in \omega_{\mathcal{G}}^{k_t} \circ \partial G(\tilde{x}^{{k_t}+1}) +\tilde{\lambda}_{k_t} A_{k_t}^T  \tilde{u}^{{k_t}+1}.
 \end{align}
Then by the upper semicontinuity of $\partial G$ and taking limits as $t \to \infty$ in \eqref{xi1}, we have from \eqref{xi2} and Remark \ref{boundedness} that
 \begin{align*}
     0 &\in \Psi_{+}' \left(x_{\mathcal{G}}^{*} \right)\circ \partial G(x^{*}).
 \end{align*}
 Thus, by using \eqref{subd_norm2} and the nonnegativity of 
 $\psi_{+}'$, we further have $x^{*}=0$, which means $\lim\limits_{t\to\infty} x^{k_t}=0$. Since $x^{k_t}\in \mathfrak{F}$ in view of Remark \ref{feasiblility} (iii), we have from the closedness of $\mathfrak{F}$ that $0\in\mathfrak{F}$, which is a contradiction to Assumption \ref{assum}(iii). Thus,  $\mathop{\lim\inf}\limits_{k \to \infty}\tilde{\lambda}_k >0$.

(ii): Using the same argument as in the proof of \cite[Theorem 3.1(ii)]{SSQTKP23}, we have that for each $k$, there exists $t_k \in \left[0,1\right]$ such that $\hat{x}^k = t_k x^k + \left(1-t_k\right) A^{\dagger}b$ satisfies $\lim\limits_{k\to\infty}t_k=1$ and $\lim\limits_{k\to\infty}\|\hat{x}^k-x^k\|=0$.
This together with
 \begin{align*}
         0 \leq \|\hat{x}^k_{\mathcal{G}} -x^k_{\mathcal{G}}\| \leq \| \left(\hat{x}^k -x^k\right)_{\mathcal{G}}\| = \|\hat{x}^k -x^k\|,
 \end{align*}
gives
\begin{equation}\label{hatx1}
\lim_{k\to\infty}\| \hat{x}_{\mathcal{G}}^k -{ x_{\mathcal{G}}^k}\| =0.
\end{equation}
Furthermore, we see from the definition of $\hat{x}^k$ and $A_k$ that
\begin{equation}\label{hatx2}
\|A_k \hat{x}^k-b^k- \tilde{v}^{k+1}\|^2 -\sigma_k\leq 0.
\end{equation}

We now   claim that
\begin{align}\label{fangshuo}
      (\omega_{\mathcal{G}}^k)^T \left(\tilde{x}^{k+1}_{\mathcal{G}} - \hat{x}^k_{\mathcal{G}} \right)  &\leq (\xi^k)^T \left(\tilde{x}^{k+1} - \hat{x}^k\right) - \frac{1}{2}\tilde{\lambda}_k \|A_k \tilde{x}^{k+1} -A_k \hat{x}^k\|^2.
 \end{align}
In fact, it follows that
\begin{align*}
       & (\omega_{\mathcal{G}}^k)^T \tilde{x}^{k+1}_{\mathcal{G}} -(\xi^k)^T \tilde{x}^{k+1} \\
       = & \; (\omega_{\mathcal{G}}^k)^T \tilde{x}^{k+1}_{\mathcal{G}} -(\xi^k)^T \tilde{x}^{k+1} +\frac{1}{2}\tilde{\lambda}_k \left(\|A_k \tilde{x}^{k+1}-b^k- \tilde{v}^{k+1}\|^2 -\sigma_k\right)\\
       \leq  \;  &(\omega_{\mathcal{G}}^k)^T \hat{x}^k_{\mathcal{G}} -(\xi^k)^T \hat{x}^k +\frac{1}{2} \tilde{\lambda}_k \left(\|A_k \hat{x}^k-b^k- \tilde{v}^{k+1}\|^2 -\sigma_k\right) \\
       &~~~-\frac{1}{2} \tilde{\lambda}_k \|A_k \tilde{x}^{k+1}-A_k \hat{x}^k\|^2\\
       \leq  \; &(\omega_{\mathcal{G}}^k)^T \hat{x}^k_{\mathcal{G}} -(\xi^k)^T \hat{x}^k - \frac{1}{2}\tilde{\lambda}_k \|A_k \tilde{x}^{k+1} -A_k \hat{x}^k\|^2,
 \end{align*}
 where the equality follows from \eqref{bianjie}, the first inequality follows from \eqref{argmin} and the second inequality holds because of \eqref{hatx2}. Rearranging the above inequality yields the assertation.
 
Now, by  the concavity of $\psi$ and the definition of $\{\omega^k_{\mathcal{G}}\}_{k\in\mathbb{N}}$ in \eqref{notation2}, we have
 \begin{align*}
     \Psi \left(x^{k+1}_{\mathcal{G}}\right) &\leq \Psi \left(x^{k}_{\mathcal{G}}\right) +(\omega^k_{\mathcal{G}})^T \left(x^{k+1}_{\mathcal{G}}- \tilde{x}^{k+1}_{\mathcal{G}} +\tilde{x}^{k+1}_{\mathcal{G}} - \hat{x}^{k}_{\mathcal{G}} +\hat{x}^{k}_{\mathcal{G}} -x^{k}_{\mathcal{G}} \right)\\
     &\leq \Psi \left(x^{k}_{\mathcal{G}}\right) + (\omega^k_{\mathcal{G}})^T \left(x^{k+1}_{\mathcal{G}}- \tilde{x}^{k+1}_{\mathcal{G}}\right) + (\xi^k)^T \left(\tilde{x}^{k+1} - \hat{x}^k\right)\\
     &~~~-\frac{1}{2}\tilde{\lambda}_k \|A_k \tilde{x}^{k+1} -A_k \hat{x}^k\|^2 + (\omega^k_{\mathcal{G}})^T  \left(\hat{x}^{k}_{\mathcal{G}} -x^{k}_{\mathcal{G}}\right),
 \end{align*}
 where the last inequality uses \eqref{fangshuo}. Rearranging the above inequality gives that

     \begin{align}\label{key ii}
    & 0 \leq \frac{1}{2} \tilde{\lambda}_k \|A_k \tilde{x}^{k+1} -A_k \hat{x}^k\|^2 \\
     &\leq \Psi \!\left(x^{k}_{\mathcal{G}}\right)\!- \!\Psi\!\left(x^{k+1}_{\mathcal{G}}\right) \!+\! (\omega^k_{\mathcal{G}})^T\! \!\left(x^{k+1}_{\mathcal{G}}\!- \!\tilde{x}^{k+1}_{\mathcal{G}}\right) \notag \!+\! (\xi^k)^T\! \!\left(\tilde{x}^{k+1}\! - \!\hat{x}^k\right)\! + \!(\omega^k_{\mathcal{G}})^T\!\!  \left(\hat{x}^{k}_{\mathcal{G}}\! -\!x^{k}_{\mathcal{G}}\right).\notag
     \end{align}
Note that the convergence of \(\left\{\Psi(x^k_{\mathcal{G}})\right\}_{k\in\mathbb{N}}\) (Proposition \ref{proposition}(i)) gives 
\[
\lim\limits_{k\to \infty}\left(\Psi(x^{k}_{\mathcal{G}})- \Psi(x^{k+1}_{\mathcal{G}})\right)=0,
\]
while \eqref{xi2} and the boundedness of \(\{\tilde{x}^{k+1}\}_{k\in\mathbb{N}}\) and  \(\{ \hat{x}^k \}_{k\in\mathbb{N}}\) imply  
\[
\lim\limits_{k\to \infty}(\xi^k)^T \left(\tilde{x}^{k+1} - \hat{x}^k\right)=0. 
\]
Combining these two displays, the boundedness of $\{A_k\}_{k\in\mathbb{N}}$ and $\{\omega^k_{\mathcal{G}}\}_{k\in\mathbb{N}}$, \eqref{hatx1}, Proposition \ref{proposition}(ii) and Theorem \ref{theorem}(i) with \eqref{key ii}, we obtain that 
\[
\lim\limits_{k\to \infty} \|A_k \left(\tilde{x}^{k+1} - \hat{x}^k\right)\|=0.
\]  
By using the boundedness of $\{A_k\}_{k\in\mathbb{N}}$ and \eqref{hatx1}, we further have
 \begin{align*}
     \lim\limits_{k\to \infty} \|A_k \left(\tilde{x}^{k+1} - x^k\right)\| =0.
 \end{align*}
This together with the use of the notations in \eqref{notation2} gives
\begin{align*}
     \lim\limits_{k \to \infty} \sqrt{\phi_{+}'\left(\left(a_i^T x^k -b_i\right)^2\right)} \left(a_i^T\left(\tilde{x}^{k+1}-x^k\right)\right)=0 \quad  ~\forall i.
 \end{align*}
Thus, combined with the boundedness of $\{x^k\}_{k\in\mathbb{N}}$ and the continuity of $\phi_{+}'$, it yields
 \begin{align*}
     \lim\limits_{k \to \infty} \phi_{+}' \left(\left(a_i^T x^k -b_i\right)^2\right) \left(a_i^T\left(\tilde{x}^{k+1}-x^k\right)\right)=0 \quad  ~\forall i.
 \end{align*}

(iii): The proof of this part is the same as that of \cite[Theorem 3.1(iii)]{SSQTKP23}, and thus is omitted herein.

(iv): Suppose, to the contrary, that $\left\{\tilde{\lambda}_k\right\}_{k\in\mathbb{N}}$ is unbounded. Then there exists a subsequence $\left\{\tilde{\lambda}_{k_t}\right\}_{k_t\in\mathbb{N}}$ such that $\lim\limits_{t\to\infty}\tilde{\lambda}_{k_t}=\infty$. In view of the boundedness of $\{x^k\}_{k\in\mathbb{N}}$ that is derived from Proposition \ref{proposition}(i), by passing to a further subsequence if necessary, we can assume $\lim\limits_{t\to \infty}x^{k_t}=x^{*}$ for some $x^{*} \in \R^n$ without loss of generality. Now, by use of the first relationship in \eqref{normal cone}, we obtain
\begin{align}\label{display}
        \frac{\xi^{k_t}}{\tilde{\lambda}_{k_t}}\in \frac{\omega^{k_t}_{\mathcal{G}}\circ \partial G(\tilde{x}^{k_t+1})}{\tilde{\lambda}_{k_t}} + A_{k_t}^T \tilde{u}^{k_t+1}.
\end{align}
Since $\partial G(\cdot)$ is contained in the closed unit ball of $\R^n$ and  $\left\{\omega^{k_t}_{\mathcal{G}}\right\}_{k_t\in\mathbb{N}}$ is bounded in view of Remark \ref{boundedness}, by passing the limit as $t \to \infty$ on both sides of \eqref{display}, we conclude from \eqref{xi2} that  $\lim\limits_{t\to\infty} A_{k_t}^T \tilde{u}^{k_t+1}=0$. This together with the definition of $\tilde v^k$ for all $k$ in \eqref{def_v^k} implies 
$       \lim\limits_{t\to\infty} \left[A_{k_t}^T \left(A_{k_t}\tilde{x}^{k_t+1}-b^{k_t}\right) - A_{k_t}^T \tilde{v}^{k_t+1}\right]=0$.
 Since $\{A_{k_t}\}_{k_t\in\mathbb{N}}$ is bounded and $\lim\limits_{t\to\infty}\tilde{v}^{k_t+1} = 0$ by virtue of  Remark \ref{boundedness} and \eqref{zzzz2} respectively, we further derive  \begin{equation}\label{ivuv1}
       \lim\limits_{t\to\infty} A_{k_t}^T \left(A_{k_t}\tilde{x}^{k_t+1}-b^{k_t}\right) = 0.
 \end{equation}
In view of \eqref{notation2}, we rewrite \eqref{ivuv1} as $      \lim\limits_{t\to\infty} \sum\limits_{i=1}^m \phi_{+}' \left(\left(a_i^T x^{k_t} -b_i\right)^2\right) \left(a_i^T \tilde{x}^{k_t+1}-b_i\right) a_i=0$.
This together with Theorem \ref{theorem}(ii) gives
\begin{align}
0&=\lim\limits_{t\to\infty}\sum\limits_{i=1}^m\phi_{+}' \left(\left(a_i^T x^{k_t} -b_i\right)^2\right) \left(a_i^T \tilde{x}^{k_t+1}-b_i\right)a_i\notag\\
&= \lim\limits_{t\to\infty}\sum\limits_{i=1}^m\phi_{+}'  \left(\left(a_i^T x^{k_t} -b_i\right)^2\right)\left(a_i^T x^{k_t}-b_i\right)a_i\notag\\ 
\label{vip11} &= \sum\limits_{i=1}^m  \phi_{+}' \left(\left(a_i^T x^{*} -b_i\right)^2\right) \left(a_i^T x^{*}-b_i\right)a_i.
\end{align}
However, by Theorem \ref{theorem}(iii), it follows that $\Phi\left((Ax^{*}-b)\circ (Ax^{*}-b)\right)=\sigma$. Together with Fact \ref{implyMFCQ}, this implies  $\sum\limits_{i=1}^m  \phi_{+}' \left(\left(a_i^T x^{*} -b_i\right)^2\right) \left(a_i^T x^{*}-b_i\right) a_i \neq 0$, which contradicts \eqref{vip11}. Thus, the sequence $\left\{\tilde{\lambda}_k\right\}_{k\in\mathbb{N}}$ is bounded.

(v): Let $x^{*}$ be any accumulation point of $\{x^k\}_{k\in\mathbb{N}}$. Then there exists a  subsequence $\{x^{k_t}\}_{k_t\in\mathbb{N}}\subset\{x^k\}_{k\in\mathbb{N}}$ such that  $\lim\limits_{t\to \infty}x^{k_t +1}=x^{*}$. This together with Proposition \ref{proposition}(ii) shows
\begin{align}\label{3.32}
     \lim\limits_{t\to \infty}\tilde{x}^{k_t +1}=x^{*}.
\end{align}
In view of the boundedness of $\{x^k\}_{k\in\mathbb{N}}$ and $\{\bar\lambda_k\}_{k\in\mathbb{N}}$, by passing to a subsequence if necessary, we have without loss of generality that 
\begin{align}\label{3.33}
     \lim\limits_{t\to \infty}x^{k_t } = \bar{x}
\end{align}
and $\lim\limits_{t\to \infty}\tilde{\lambda}_{k_t}=\lambda_{*}$ for some $\bar x\in \R^n$ and $\lambda_*\in \R$, respectively. Since $x^{*}$ is an accumulation point of $\{x^k\}_{k\in\mathbb{N}}$, in view of Theorem \ref{theorem}(i) and (iii), one has that $x^*$ and $\lambda_*/2$ satisfy the former two conditons \eqref{kkt1} and \eqref{kkt2} in Definition \ref{KKT}. That is,
$\Phi\left((Ax^{*}-b) \circ (Ax^{*}-b)\right) \leq \sigma$ and $\lambda_{*} \left(\Phi((Ax^{*}-b)\circ (Ax^{*}-b))-\sigma\right)=0$. Now, to complete the proof, it suffices to show that $\left(x^{*},\frac{\lambda_{*}}{2}\right)$ also satisfies \eqref{kkt3}.

In order to do so, we first have from  Theorem \ref{theorem}(ii) that for each $i$, 
\begin{align}\label{relationii}
\displaystyle&\lim\limits_{t\to\infty}\phi_{+}' \left(\left(a_i^T x^{k_t} -b_i\right)^2\right) \left(a_i^T \tilde{x}^{k_t+1}-b_i\right)\notag\\
=  \displaystyle&\;\lim\limits_{t\to\infty}\phi_{+}' \left(\left(a_i^T x^{k_t} -b_i\right)^2\right) \left(a_i^T x^{k_t}-b_i\right).
\end{align}
 Define $I=\left\{i:\phi_{+}'  \left(\left(a_i^T \bar{x} -b_i\right)^2\right) > 0\right\}$. Then  for each $i\in I$, it follows from \eqref{3.32}-\eqref{relationii} and the continuity of $\phi_{+}'$ that 
 \begin{equation}\label{equality_limited points}
 a_i^T x^{*}=a_i^T \bar{x}.
 \end{equation}
 On the other hand, for each $i\notin I$ (i.e., $\phi_{+}' \left(\left(a_i^T \bar{x} -b_i\right)^2\right)=0$), we claim that
 \begin{equation*}
 \phi_{+}' \left(\left(a_i^T x^{*} -b_i\right)^2\right)=0.
 \end{equation*}
 In fact, if on the contrary that there exists $i_0 \notin I$ such that $\phi_{+}' \left(\left(a_{i_0}^T x^{*} -b_{i_0}\right)^2\right)>0$, then by using the concavity of $\phi$, one has that $\left(a_i^T \bar{x}-b_i\right)^2$ is a maximizer of $\phi$ for all $i\notin I$. Thus, the following two inequalities hold:
\begin{align*}
      \phi\left(\left(a_{i_0}^T x^{*} -b_{i_0}\right)^2\right)&< \phi\left(\left(a_{i_0}^T \bar{x} -b_{i_0}\right)^2\right), \quad \text{and}\\
      \phi\left(\left(a_{i}^T x^{*} -b_{i}\right)^2\right)&\leq \phi\left(\left(a_{i}^T \bar{x} -b_{i}\right)^2\right)~\forall i \notin I\cup \{i_0\}.
\end{align*}
In view of above display, we further have 
\begin{align}
     &\Phi \left((Ax^{*}-b)\circ (Ax^{*}-b)\right) =\sum\limits_{i=1}^m  \phi\left(\left(a_{i}^T x^{*} -b_{i}\right)^2\right) \notag\\
     &=\phi\left(\left(a_{i_0}^T x^{*} -b_{i_0}\right)^2\right) +\sum\limits_{i\neq i_0, i\notin I} \phi\left(\left(a_{i}^T x^{*} -b_{i}\right)^2\right) +\sum\limits_{i\in I} \phi\left(\left(a_{i}^T \bar{x} -b_{i}\right)^2\right)\notag\\
     &\leq \phi \left(\left(a_{i_0}^T x^{*} -b_{i_0}\right)^2\right) + \sum\limits_{i\neq i_0, i\notin I} \phi \left(\left(a_{i}^T \bar{x} -b_{i}\right)^2\right) +\sum\limits_{i\in I} \phi\left(\left(a_{i}^T \bar{x} -b_{i}\right)^2\right) \notag\\
     &< \phi \left(\left(a_{i_0}^T \bar{x} -b_{i_0}\right)^2\right)+ \sum\limits_{i\neq i_0, i\notin I} \phi \left(\left(a_{i}^T \bar{x} -b_{i}\right)^2\right) +\sum\limits_{i\in I} \phi \left(\left(a_{i}^T \bar{x} -b_{i}\right)^2\right) \notag\\
     &=\sum\limits_{i=1}^m  \phi \left(\left(a_{i}^T \bar{x} -b_{i}\right)^2\right) =\Phi \left((A\bar{x}-b)\circ (A\bar{x}-b)\right), \notag
 \end{align}
where the second equality follows from \eqref{equality_limited points}.
However, since $x^{*}$ and $\bar{x}$ are both accumulation points of $\{x^k\}$, the above display contadicts 
\[\Phi \left((Ax^{*}-b)\circ (Ax^{*}-b)\right) = \Phi \left((A\bar{x}-b)\circ (A\bar{x}-b)\right) =\sigma\]
which is deduced from Theorem \ref{theorem}(iii). Therefore, we obtain the claim and conclude that
\begin{align}\label{relation conclusion}
    \begin{cases}
    a_i^T x^{*} =a_i^T \bar{x}, & {\rm if}\  \phi_{+}' \left(\left(a_i^T \bar{x} -b_i\right)^2\right)>0,\\
    \phi_{+}' \left(\left(a_i^T \bar{x} -b_i\right)^2\right) = \phi_{+}'  \left(\left(a_i^T x^{*} -b_i\right)^2\right) =0, & {\rm if}\  \phi_{+}'  \left(\left(a_i^T \bar{x} -b_i\right)^2\right) =0.
    \end{cases}
\end{align}

Now, with \eqref{normal cone}, the definition of $\tilde{v}^{k+1}$ in \eqref{def_v^k} and the notations in \eqref{notation2}, one has 
\begin{align*}
     \xi^{k_t} &\in \omega_{\mathcal{G}}^{k_t} \circ \partial G(\tilde{x}^{k_t+1}) +\tilde{\lambda}_{k_t} A_{k_t}^T  \tilde{u}^{k_t+1}\\
     &=\Psi_{+}' (x^{k_t}_{\mathcal{G}})\circ \partial G(\tilde{x}^{k_t+1})+\tilde{\lambda}_{k_t} A_{k_t}^T \left(A_{k_t}\tilde{x}^{k_t+1} -b^{k_t} -\tilde{v}^{k_t+1}\right)\\
     &=\Psi_{+}' \left(\tilde{x}^{k_t+1}_{\mathcal{G}} +x^{k_t}_{\mathcal{G}} - \tilde{x}^{k_t+1}_{\mathcal{G}}\right) \circ \partial G(\tilde{x}^{k_t+1}) \\
     &~~~+\tilde{\lambda}_{k_t} \left(\sum\limits_{i=1}^m \phi_{+}' \left(\left(a_i^T x^{k_t} -b_i\right)^2\right) \left(a_i^T \tilde{x}^{k_t+1}-b_i\right)a_i -A^T_{k_t} \tilde{v}^{k_t+1}\right).
\end{align*}
Recall \eqref{3.32}, \eqref{3.33} and $\lim\limits_{t\to \infty}\tilde{\lambda}_{k_t} = \lambda_{*}$. Then taking the limit as $t \to \infty$ on the leftmost and rightmost sides of above display, we obtain from Proposition \ref{proposition}(iii), \eqref{relation conclusion}, the boundedness of $\{A_k\}_{k\in\mathbb{N}}$, and the definition of $\tilde{v}^{k+1}$ that
\begin{align*}
        0\in \Psi_{+}' (x^{*}_{\mathcal{G}})\circ \partial G(x^{*})+\lambda_{*} \sum\limits_{i=1}^m \phi_{+}'  \left(\left(a_i^T x^{*} -b_i\right)^2\right) \left(a_i^T x^{*} -b_i\right)a_i
\end{align*}
This implies that \eqref{kkt3} holds when $(x,\lambda)$ is replaced by $\left(x^{*},\frac{\lambda_{*}}{2}\right)$. The proof is therefore completed.
\end{proof}

\section{Numerical experiments}\label{section5}
 In this section,  we show the performance of Algorithm $\IR$ via numerical experiments. Our codes are written and implemented in MATLAB 2019b, and all numerical experiments are performed on a 64-bit PC equipped with an Intel(R) Xeon(R) Silver 4210 CPU (2.20GHz) and 32GB of RAM.

 We first illustrate the notations in Tables~\ref{tableforGIRIRiter}-\ref{Comparision of the performance 3}.
 The minimum and maximum values of the $ \rm residuals= (\Psi(\vartheta^k_{\mathcal{G}})-\sigma)/{\sigma}$
 ($\rm Res_{\min}$ and $\rm Res_{\max}$, respectively) are  recorded, where $\vartheta^k$ denotes the approximate sparse solution generated by the corresponding algorithm. The objective function value of the recovered vector is recorded as $\rm Fval$. We also record three kinds of data refer to the instances being ``successfully solved'': the average number of inner iterations ($\rm Iter_s$), the average CPU time ($\rm CPU_s$), and the average recovery error ($\rm RecErr_s$). If instances that fail to meet ``successfully solved" are  recorded as $\rm Iter_f$, $\rm CPU_f$ and $\rm RecErr_f$. We call a random instance is ``successfully solved'' if the recovery error meets the criterion $\frac{\|\vartheta^k-x_{\rm orig}\|}{\max\{\|x_{\rm orig}\|,1\}}\leq 0.01$, in which $x_{\rm orig}$ is generated through MATLAB commands in the following Subsection 5.1.

\subsection{Importance of new problem and algorithm}\label{improt-sec}

  The introduction of a group sparsity structure elevates the importance of problem \eqref{Problem}   and Algorithms \ref{algorithm1} and \ref{Algorithm2}, named  as ${\IR}_{\rm ADMM}$, which form a central focus of this subsection.

  We consider the following constrained optimization problem:
\begin{equation}\label{problem in experiment}
    \begin{aligned}
        \min_{x\in \R^n}& \sum_{i=1}^q \log\left(1+\|x_{\mathcal{G}_i}\| /\epsilon\right), \quad\;
        {\rm s.t.} \quad\; \sum_{i=1}^m \log \left( 1+(a_i^T x-b_i)^2/\delta^2 \right) \leq \sigma,
    \end{aligned}
\end{equation}
where $m\ll n$, $\epsilon >0$, $\delta >0$, $b\in \R^m$ and $\sigma >0$. The row vectors $a_1,a_2,\cdots ,a_m$ of matrix $A$ are linearly independent. Here, $\psi$ and $\phi$ correspond to the $\log$ penalty function and Cauchy loss function, respectively. We can see that this model incorporates a log penalty function \cite{ChenLuPong2016} in the objective function, which can effectively induce variable sparsity while enhancing model interpretability, reducing complexity, and stabilizing the optimization process. By integrating the Cauchy loss function \cite{LiLuDongTao2018} into the constraints, it significantly improves the model's robustness and strengthens resistance to outliers and noise. 

We note that \eqref{problem in experiment} is a special case of problem \eqref{Problem}, which 
satisfies  Assumptions \ref{assum}.
Thus,  every accumulation point of the generated sequence $\{ x^k\}_{k\in\mathbb{N}}$ by Algorithm \ref{algorithm1} is a stationary point, as guaranteed by Theorem \ref{theorem}. 

For the numerical comparsion, we generate an $m\times n$ matrix $A$ with independent and identically distributed (i.i.d.) standard Gaussian entries. Next, we generate an original block-sparse signal $x_{\rm{orig}}\in \R^n$ using the following MATLAB commands:
\begin{verbatim}
        I = randperm(n/J); I = I(s+1 : end); 
        x0 = randn(J,n/J); x0(:, I) = 0; 
        x_{orig} = reshape(x0, n, 1);
\end{verbatim}
Here, 
${J}$ denotes the size of each block, and $s$ is the number of nonzero blocks. We further define $\eta = 0.005*{\rm{randn}}(m, 1)$, where the $\rm{noise}$ entries follow an i.i.d standard Cauchy distribution, and set $b = Ax_{\rm orig} + \eta$. Finally, we specify $J=2$, $\epsilon=0.1$ and $\sigma=1.2 \sum\limits_{i=1}^{m}\log \left(1+(\eta_i)^2/\delta^2 \right)$ with $\delta=0.05$. 

In the numerical experiments, we set the three-element tuple $(m,n,s)=(540i,2560i,80i)$ for $i\in\{2,4,6,8,10\}$.  Based on the generation of $x_{\rm orig}$, we have ${\rm nnz}(x_{{\rm orig}})=80*2i$. For each $i$, 30 random instances are generated by the above commands. We report the mean value of $L$ ($\rm L$), the average CPU time for computing $L$ (${\rm Time_L}$), the average  CPU time for computing $A^{\dagger}b$ (${\rm Time_{slater}}$), and the average CPU time for performing $\rm QR$ decomposition of $A^T$ ($\rm Time_{QR}$) in Table \ref{Record of  average value and  average CPU time}.

 \begin{table}[h]
    \centering
    \caption{Record of  average value and  average CPU time}\label{Record of  average value and  average CPU time}
    \begin{tabular}{lcccc} 
        \toprule  
        \multicolumn{1}{c}{$(m,n,s)$} & $\rm L$ & $\rm Time_{L}$& $\rm Time_{QR}$ & $\rm Time_{slater}$\\
        \midrule  
        $(1080,5120,160)$ & 1.08e+04  &  0.2  &    0.4  &  0.0\\
        $(2160,10240,320)$ & 2.18e+04  &   1.0  &    1.3  &  0.0  \\
        $(3240,15360,480)$ & 3.26e+04  &  2.5  &    2.9  &  0.0 \\
        $(4320,20480,640)$ & 4.35e+04  &  5.2  &    5.2  &  0.1  \\
       $(5400,25600,800)$  & 5.44e+04  &  6.6  &    8.5  &  0.1  \\
        \bottomrule  
    \end{tabular}
\end{table}

Note that if $q=n$ in problem \eqref{problem in experiment},
then it reduces to 
\begin{equation}\label{SSQmodleinexperiment}
    \begin{aligned}
        \min_{x\in \R^n}& \sum_{i=1}^n \log\left(1+ |x_i| /\epsilon\right),\quad
        {\rm s.t.}\quad \sum_{i=1}^m \log \left( 1+(a_i^T x-b_i)^2/\delta^2 \right) \leq \sigma.
    \end{aligned}
\end{equation}
This problem is  a special case of problem \eqref{problem for ssq} without group structure, which is from \cite{SSQTKP23}.
Sun and Pong \cite{SSQTKP23}  also propose the double iterative reweighted algorithm to solve problem \eqref{problem for ssq}, whose subproblem is also  can be solved by ADMM.
These strategies are named as ${\YIR}_{\rm ADMM}$  for short.

For  ${\IR}_{\rm ADMM}$ and ${\YIR}_{\rm ADMM}$, we have the following common settings.  The initial point $x^0=A^{\dagger}b$ is computed using the following MATLAB commands:
\begin{verbatim}
		[Q,R] = qr(A',0); xfeas = Q*(R'\b);
\end{verbatim}
We set $\tau_k=\max\left\{5^{-k-1}, 10^{-8}\right\}$ and $\mu_k=\max \left\{1.2^{-k-1},10^{-8}\right\}$ for Algorithm~\ref{algorithm1}.
The outer iteration terminates when the criterion $\frac{\|x^{k+1}-x^k\|}{\max\left\{\|x^k\|,1\right\}}\leq 10^{-4}$ is met.

The other parameters of ${\YIR}_{\rm ADMM}$ for its inner problem is what the reference \cite{SSQTKP23} proposed.
In the implementation of  ${\IR}_{\rm ADMM}$,
at each iteration $k$, 
  we set 
$\beta=\bar{L}^{-\frac{1}{2}}$, $\gamma=\frac{0.99(1+\sqrt{5})}{2}$,  $\rho$ and  $\bar{L}$ as in Remark \ref{remarkforbanzhengding} for Algorithm \ref{Algorithm2}.  
Note that ${\YIR}_{\rm ADMM}$ use the  standard $\ell_2$-norm square as distance function. Thus, we choose $f$ in Definition \ref{def1} to be $f(x)=\|x\|^2$. This implies \eqref{zzzz1} equals to
\begin{equation*}
	{\rm dist}\left(0,\partial \left(\omega_{\mathcal{G}}^k\circ G\right)(\tilde{x}^{k+1})+A_{k}^TN_{ U^k}(\tilde{u}^{k+1})\right)\leq \sqrt{\epsilon_{k}},
\end{equation*}
where $\epsilon_k= \min\left\{ \bar{\epsilon}_k, \tau_k \bar{\Gamma_l}, \gamma\beta \bar{\epsilon}_k, \gamma \beta \tau_k(\|z^{k, l+1}\|+1) \right \}$.

The numerical results are reported in Tables \ref{tableforGIRIRnnz} and \ref{tableforGIRIRiter}.
In Table \ref{tableforGIRIRnnz},
${\rm Success(\%)}$ and ${\rm nnz(\%)}$ are used to  denote
  the success rate and the percentage of sparsity successfully recovered, respectively.
  From the value of ${\rm nnz}$,   we can see that problem \eqref{problem in experiment} can obtain more sparser solution point than problem~\eqref{SSQmodleinexperiment}.

\begin{table}[h]
  \centering
  \caption{Vector sparsity recovery success rate and recovery degree proportion}
  \label{tableforGIRIRnnz}
  \begin{tabular}{lcccc}
    \toprule
    \multicolumn{1}{c}{Dimension and sparsity} & \multicolumn{2}{c}{${\IR}_{\rm ADMM}$} & \multicolumn{2}{c}{${\YIR}_{\rm ADMM}$ in \cite{SSQTKP23}} \\
    \cmidrule(lr){2-3} \cmidrule(lr){4-5} 
    \multicolumn{1}{c}{$(m,n,s)$} & $\rm Success(\%)$ & $\rm nnz(\%)$ & $\rm Success(\%)$ & $\rm nnz(\%)$ \\
    \midrule
    $(1080,5120,160)$ & 100 & 100 & 100 & 63\\
    $(2160,10240,320)$ & 100 & 100 & 100 & 17\\
    $(3240,15360,480)$ & 100 & 100 & 100 & 3 \\
    $(4320,20480,640)$ & 100 & 100 & 100 & 20 \\
    $(5400,25600,800)$ & 100 & 100 & 100 & 0 \\
    \bottomrule
  \end{tabular}
\end{table}

We present the more comparation on ${\IR}_{\rm ADMM}$ and ${\YIR}_{\rm ADMM}$     in Table~\ref{tableforGIRIRiter}.
It can be observed that ${\IR}_{\rm ADMM}$ consistently requires less CPU time and a smaller value of $\rm Fval$ than ${\YIR}_{\rm ADMM}$. The number of iterations needed for Algorithm ${\IR}_{\rm ADMM}$ is fewer than for ${\YIR}_{\rm ADMM}$.  The results indicate that  ${\IR}_{\rm ADMM}$ outperforms ${\YIR}_{\rm ADMM}$ in preserving the sparsity of the original vector across all cases.
\begin{table}[h]
  \centering
  \caption{Numerical comparison  of ${\IR}_{\rm ADMM}$ and ${\YIR}_{\rm ADMM}$}
  \label{tableforGIRIRiter}
  \resizebox{\textwidth}{15mm}{
  \begin{tabular}{lcccccccccccc}
    \toprule
    \multicolumn{1}{c}{Dimension and sparsity} & \multicolumn{6}{c}{${\IR}_{\rm ADMM}$} &  \multicolumn{6}{c}{${\YIR}_{\rm ADMM}$ in \cite{SSQTKP23}} \\
    \cmidrule(lr){2-7} \cmidrule(lr){8-13}
    \multicolumn{1}{c}{$(m,n,s)$} & $ \rm Iter_s $ & $\rm CPU_s$ & $\rm Fval$ & $\rm RecErr_s$ & $\rm Res_{min}$ & $\rm Res_{max}$ & $ \rm Iter_s $ & $\rm CPU_s$ & $\rm Fval$ & $\rm RecErr_s$ & $\rm Res_{min}$ & $\rm Res_{max}$ \\
    \midrule
    $(1080,5120,160)$ & 1576 & 10.2 & 4.0e+02 &  3.2e-04 & -3.9e-02 & -2.2e-03 & 2835 & 19.4 & 6.2e+02 & 3.3e-04 & -5.3e-03 & 6.0e-04\\
    $(2160,10240,320)$ & 2231 &  51.8 & 7.9e+02 & 2.4e-04 & -6.8e-03 & -1.4e-05 & 3812   & 97.0 & 1.2e+03 & 2.4e-04 & -1.5e-03 & 1.7e-05\\
    $(3240,15360,480)$ & 2870   & 143.7 & 1.2e+03 & 1.9e-04 & -2.6e-03 & -3.6e-06 & 4635   & 246.1 & 1.9e+03 &  1.9e-04 & -6.0e-04 & -6.3e-06\\
    $(4320,20480,640)$ & 3422  & 298.2  & 1.6e+03 & 1.7e-04 & -1.8e-03 & -2.2e-05 & 5400   & 525.4  & 2.5e+03 & 1.7e-04 & -3.2e-04 & -7.1e-06 \\
    $(5400,25600,800)$ & 3939 & 519.7 & 2.0e+03 & 1.5e-04 & -1.0e-03 & -1.1e-05 & 6212 & 947.6 & 3.1e+03 & 1.5e-04 & -3.3e-04 & -4.4e-06\\
    \bottomrule
  \end{tabular}}
\end{table}


\subsection{Algorithm comparision}
 
In this subsection, we first apply ${\IR}_{\rm ADMM}$ and $\textbf{SPA}$ to solve the following gourp sparse optimization problem
\begin{equation}
\label{problem in experiment-2}
    \begin{aligned}
        \min_{x\in \R^n}& \sum_{i=1}^q \log\left(1+\|x_{\mathcal{G}_i}\| /\epsilon\right), \quad\;
        {\rm s.t.} \quad\; \|Ax-b\| \leq \sqrt{\sigma}.
    \end{aligned}
\end{equation}
This problem is a special case of both \eqref{Problem} and \eqref{proble-pan}. It reduces to \eqref{Problem} when $\phi = I$, and to \eqref{proble-pan} because its objective function can be reformulated as a capped folded concave function (see Appendix for details on this transformation and the proximal operator computation). Therefore, the $\mathbf{SPA}$ method can be applied to solve it.

In the implementation, we use the same way to generate $A$, $b$, $\sigma$ in the Subsection \ref{improt-sec}, and we set $\epsilon=0.3$.
 For \textbf{SPA}, we set the parameter as $\nu=\epsilon \left(e^{C}-1\right)$, $n_i=J=2$, $x^0=1_n$,  $\lambda_0=40$, $\mu_0=\epsilon_0=1$, $\rho=3$, $\theta=\frac{1}{\rho}$, $M=3$, where $\nu$ and $C$ are defined in Remark \ref{reformulation of prob_numexperiment}.  We set  the feasible point the same as that in ${\IR}_{\rm ADMM}$.
The entire algorithm that we set terminates when 
\begin{equation*}
    \max\left\{ \left(\|Ax-b\|^2 -\sigma \right)_{+},0.01\epsilon_k \right\} \leq 10^{-6}
\end{equation*}
is satisfied.

The corresponding numerical reuslts are presented in Table \ref{Comparision of the performance 2}.  Here we can see that although \textbf{SPA} requires less CPU time than ${\IR}_{\rm ADMM}$  to solve the problem \eqref{problem in experiment-2}, it underperforms  ${\IR}_{\rm ADMM}$ in  vector recovery, sparsity preservation and the value of $\rm Fval$.

\begin{table}[h]
  \centering
  \caption{The first numerical records of ${\IR}_{\rm ADMM}$ and $\mathbf{SPA}$}
    \label{Comparision of the performance 2}
  \resizebox{\textwidth}{15mm}{
  \begin{tabular}{lcccccccccccc}
    \toprule
    \multicolumn{1}{c}{Dimension and sparsity} & \multicolumn{6}{c}{${\IR}_{\rm ADMM}$} & \multicolumn{6}{c}{$\mathbf{SPA}$ in \cite{LiliPan2021}} \\
    \cmidrule(lr){2-7} \cmidrule(lr){8-13} 
    \multicolumn{1}{c}{$(m,n,s)$} & $ \rm Iter_s $ & $\rm CPU_s$ & $\rm Fval$ & $\rm RecErr_s$ & $\rm Res_{min}$ & $\rm Res_{max}$ & $ \rm Iter_f $ & $\rm CPU_f$ & $\rm Fval$ & $\rm RecErr_f$ & $\rm Res_{min}$ & $\rm Res_{max}$ \\
 \midrule
    $(1080,5120,160)$ & 1848 & 11.5 & 2.5e+02 & 3.3e-04 & -10.0e-01 & -4.7e+00 & 47& 3.0 & 9.3e+02 & 8.9e-01 & -4.0e-05 & 4.2e-7\\
    $(2160,10240,320)$ & 2144  &   47.9 & 5.0e+02 & 2.4e-04 & -10.0e-01 & 4.3e+00 & 32  & 6.5 & 1.9e+03  & 8.9e-01  & -5.1e-05 & 1.9e-07 \\
    $(3240,15360,480)$ & 2137  &  102.9  & 7.5e+02 & 1.9e-04  & -9.8e-01 & 9.1e+00 &  32  & 14.0  & 2.8e+03 & 8.9e-01  & -1.9e-04 & 4.4e-07\\
    $(4320,20480,640)$ & 1885  &  159.2  & 9.9e+02 & 1.7e-04  & -10.0e-01 & 8.6e-01 & 30  & 23.0  & 3.7e+03 & 8.9e-01  & -1.5e-05 & 4.0e-07\\
    $(5400,25600,800)$ & 2020  &  260.7  & 1.2e+03 & 1.5e-04  & -10.0e-01 & 3.8e+00 & 29  & 34.0 & 4.6e+03 & 8.9e-01  & -1.9e-05 & 2.7e-07\\
    \bottomrule
  \end{tabular}}
\end{table}

Next, we conduct a comparative numerical study to evaluate the performance of the proposed  ${\IR}_{\rm ADMM}$  against \textbf{SPA} in solving problem \eqref{problem in experiment}, which is
\begin{equation}\label{prob:SPA-GIR}
    \begin{aligned}
        \min_{x\in \R^n}& \sum_{i=1}^q \log\left(1+\|x_{\mathcal{G}_i}\| /\epsilon\right), \quad\;
        {\rm s.t.} \quad\; \sum_{i=1}^m \log \left( 1+(a_i^T x-b_i)^2/\delta^2 \right) \leq \sigma.
    \end{aligned}
\end{equation}
The objective of this experiment is to benchmark our approach against a recognized method, assessing key performance indicators such as solution accuracy, and computational stability.

In the implementation, we also use the same way to generate $A$, $b$, $\sigma$ as that in Subsection \ref{improt-sec}. The setting of $\IR_{\rm ADMM}$ is also what Subsection \ref{improt-sec} proposed. The parameter settings of \textbf{SPA} are the same as those used when solving \eqref{problem in experiment-2}.
Moreover, we set the stopping criterion of  \textbf{SPA}  as
\begin{equation*}
    \max\left\{ \left(\sum_{i=1}^m \log \left( 1+(a_i^T x-b_i)^2/\delta^2 \right) -\sigma \right)_{+},0.01\epsilon_k \right\} \leq 10^{-6}.
\end{equation*}
Numerical results show that, for our Algorithm ${\IR}_{\rm ADMM}$, the $\rm Success (\%)$ and $\rm nnz(\%)$ recorded still satisfy $100\%$, while \textbf{SPA} still fails to achieve recovery and cannot guarantee sparsity. The other experimental data are recorded in Table \ref{Comparision of the performance 3}, which also demonstrates the excellent performance of ${\IR}_{\rm ADMM}$. 

\begin{table}[h]
  \centering
  \caption{The second numerical records of ${\IR}_{\rm ADMM}$ and $\mathbf{SPA}$}
    \label{Comparision of the performance 3}
  \resizebox{\textwidth}{15mm}{
  \begin{tabular}{lcccccccccccc}
    \toprule
    \multicolumn{1}{c}{Dimension and sparsity} & \multicolumn{6}{c}{${\IR}_{\rm ADMM}$} & \multicolumn{6}{c}{$\mathbf{SPA}$ in \cite{LiliPan2021}} \\
    \cmidrule(lr){2-7} \cmidrule(lr){8-13} 
    \multicolumn{1}{c}{$(m,n,s)$} & $ \rm Iter_s $ & $\rm CPU_s$ & $\rm Fval$ & $\rm RecErr_s$ & $\rm Res_{min}$ & $\rm Res_{max}$ & $ \rm Iter_f $ & $\rm CPU_f$ & $\rm Fval$ & $\rm RecErr_f$ & $\rm Res_{min}$ & $\rm Res_{max}$ \\
   \midrule
    $(1080,5120,160)$ & 2431 & 15.7 & 2.5e+02 & 3.4e-04 & -5.0e-03 & -2.5e-05 & 23& 1.5 & 9.3e+02 & 8.9e-01 & -8.1e-06 & -3.5e-11\\
    $(2160,10240,320)$ &  3497 & 80.1 & 4.9e+02 & 2.4e-04 & -5.3e-04 & 7.6e-08 & 22 & 4.6 & 1.9e+03 & 8.9e-01 & -9.2e-06& -2.4e-11\\
    $(3240,15360,480)$ & 4425 & 221.6 & 7.5e+02 & 1.9e-04& -2.0e-04& -2.3e-05 & 23&  9.9 & 2.8e+03 & 8.9e-01& -1.2e-05& -4.8e-10\\
    $(4320,20480,640)$ & 5207& 451.6& 9.9e+02 & 1.6e-04 & -2.6e-03& -2.2e-07 & 22& 16.0& 3.7e+03 & 8.9e-01& -1.5e-05& -9.5e-10 \\
    $(5400,25600,800)$ & 5945& 781.0& 1.2e+03 & 1.5e-04& -3.1e-04& -2.5e-07 & 22& 26.0& 4.6e+03 &8.9e-01& -1.8e-05& -7.3e-12\\
    \bottomrule
  \end{tabular}}
\end{table}


\section{Conclusion}
In this work, we propose the first application of the  iteratively reweighted algorithm to solve a class of group sparse optimization problems. Both the objective function and the constraint of our problem involve non-convex functions, which broadens the scope of group sparse-related research models and enhances the robustness of the model. By appropriately reformulating the problem, the non-convex problem is transformed into a sequence of convex subproblems, which are then approximated using the well-established convex solver ADMM. In the algorithm design, we design the termination criteria of subproblem solver  by using Bregman distance rather than European distance. Furthermore, the update rule for the next iteration point such that all iteration points generated by the proposed algorithm remain within the feasible region.

However, due to the specific structure of the constructed subproblems, it is challenging to  ensure that $x^{k+1} - x^k\rightarrow 0$. Additionally, analyzing the convergence rate of the algorithm for our research remains non-trivial. Consequently, future research directions may include: (1) analyzing the convergence rate of the proposed algorithm, (2) extending the application of the algorithm to practical, concrete scenarios, and (3) design alternative algorithms for solving the problem model in this paper, leveraging the explicit expression of the proximity operator of the log-sum  penalty function for algorithm development.

\section*{Appendix}
Inspired by \cite[Section 6.1]{ZhangYongle2023}, we first establish that problems \eqref{problem in experiment-2} and \eqref{prob:SPA-GIR} are equivalent to an optimization problem with an additionally imposed bounded feasible set.

\begin{lemma}\label{The form capped folded function for our problem}
	Under Assumption \ref{assum}, the solution set of problems \eqref{problem in experiment-2} and \eqref{prob:SPA-GIR} are bounded.
\end{lemma}	

 \begin{proof}
Note that $A^{\dagger}b$ is a feasible point of problems \eqref{problem in experiment-2} and \eqref{prob:SPA-GIR}. In view of the definition of $G$ in Section 1, we set $\tilde{x}_{\mathcal{G}}=G(A^{\dagger}b)$ and define $C=\sum\limits_{i=1}^q \log\left(1+\frac{(\tilde{x}_{\mathcal{G}})_{i}}{\epsilon}\right).$ The level-boundedness of the objective function in problems \eqref{problem in experiment-2} and \eqref{prob:SPA-GIR} implies that its solution set (denoted by $S$) is nonempty. For any $x\in S$, the following inequality holds:
 \[\sum_{i=1}^q \log\left(1+\frac{({x}_{\mathcal{G}})_{i}}{\epsilon}\right) \leq \sum_{i=1}^q\log\left(1+\frac{(\tilde{x}_{\mathcal{G}})_{i}}{\epsilon}\right).\] Given that $\log\left(1+\frac{(\cdot)}{\epsilon}\right)$ is nondecreasing over $\R_+$, we obtain:
\begin{equation*}
	C\geq \sum_{i=1}^q \log\left(1+\frac{({x}_{\mathcal{G}})_{i}}{\epsilon}\right) \geq \max_{1\leq i\leq q}\log\left(1+\frac{({x}_{\mathcal{G}})_{i}}{\epsilon}\right)
	=\log\left(1+\frac{\|{x}_{\mathcal{G}}\|_{\infty}}{\epsilon}\right).
\end{equation*}
This implies $\|{x}_{\mathcal{G}}\|_{\infty} \leq \epsilon \left(e^C-1\right)$. The proof is completed by invoking the definition of ${x}_{\mathcal{G}}$.
\end{proof}

\begin{remark}\label{reformulation of prob_numexperiment}
 Lemma \ref{The form capped folded function for our problem} indicates $\sum\limits_{i=1}^q \log\left(1+\|x_{\mathcal{G}_i}\| /\epsilon\right)$ is equivalent to the following form:
	\begin{equation}\label{reproblem in experiment}
		\begin{aligned}
			 \sum_{i=1}^q \psi^{\rm CapLog}((x_{\mathcal{G}})_i)
		\end{aligned}
	\end{equation}
where $\psi^{\rm CapLog}$ is the capped function defined as:
\begin{equation}\label{capped function form for our problem}
	\psi^{\rm CapLog}(t)=
	\begin{cases}
		\log \left(1+\frac{|t|}{\epsilon}\right),& {\rm if}\ |t|\in [0,\nu),\\
		C, & {\rm if}\ |t|\in [\nu,\infty),
	\end{cases}
\end{equation}
 with $\nu=\epsilon \left(e^C-1\right)$,  $C=\sum\limits_{i=1}^q \log \left(1+\frac{(\tilde{x}_{\mathcal{G}})_{i}}{\epsilon}\right)$ and $\tilde{x}_{\mathcal{G}}=G(A^{\dagger}b)$.
\end{remark}

Our next goal is to derive the closed-form expression of the proximal mapping for the objective function in the reformulated problem \eqref{reproblem in experiment}. We first present the following lemma:
\begin{lemma}\label{proximal mapping for our problem}
	Let $C$ be defined as in \eqref{capped function form for our problem}. Define $f:\R^n \to \R$ to be
	\begin{align*}
	  f(x)=\begin{cases}
		\log \left(1+\frac{\|x\|}{\epsilon}\right), & {\rm if}\ \|x\|\in [0,\nu),\\
		C, & {\rm if}\ \|x\|\in [\nu,\infty),
	\end{cases}
	\end{align*}
	 where $\epsilon>0$ is a given parameter. The proximal mapping of $f$ with scaling parameter $\lambda>0$ at a point $x\in\R^n$ is:
	\begin{align}\label{proximalM_logP}
		\prox_{\lambda f}(x)=\begin{cases}
			0, & {\rm if}\ \|x\|=0,\\
			\frac{x}{\|x\|}\prox_{\lambda \psi^{\rm CapLog}}(\|x\|), & {\rm otherwise},\\
		\end{cases}
	\end{align}
	where $\psi^{\rm CapLog}$ is defined in \eqref{capped function form for our problem}.
\end{lemma}

\begin{proof}
If $x=0$, the definition of the proximal mapping (given in \eqref{eq:def-prox}) directly implies $\prox_{\lambda f}(x) = 0$.

For $x\neq 0$, we use the definition of the proximal mapping in \eqref{eq:def-prox} to obtain:
\begin{align*}
	\prox_{\lambda f}(x)& = \argmin_{y\in \R^n}\left\{f(y)+\frac{\|x-y\|^2}{2\lambda}\right\}\\
	& = \argmin_{y\in \R^n}\left\{f(y)+\frac{\|x\|^2+\|y\|^2 - 2\left\langle x, y \right\rangle}{2\lambda}\right\}\\
	& = \frac{x}{\|x\|} \argmin_{\|y\| \in \R_+} \left\{\psi^{\rm CapLog}(\|y\|)+\frac{(\|x\| - \|y\|)^2 }{2\lambda}\right\}\\
    &=  \frac{x}{\|x\|} \argmin_{\omega \in \R_+} \left\{\psi^{\rm CapLog}(\omega)+\frac{(\|x\| - \omega)^2 }{2\lambda}\right\}\\
	&=  \frac{x}{\|x\|} \argmin_{\tilde\omega \in \R} \left\{\psi^{\rm CapLog}(\tilde\omega)+\frac{(\|x\| - \tilde\omega)^2 }{2\lambda}\right\}
	 = \frac{x}{\|x\|}  \prox_{\lambda \psi^{\rm CapLog}}(\|x\|),
\end{align*}
where the third equality holds because the minimizer of the optimization problem in the RHS of the second equation is achieved only if the vectors $x$ and $y$ are in the same direction, the fifth equality follows from the same reason as that in the third one. This completes the proof.
\end{proof}

Next, we present the closed form of $\prox_{\lambda f}(x)$ (from \eqref{proximalM_logP}) for the case $\|x\| \neq 0$, which is established in the following lemma:

\begin{lemma}
     Consider the function $f$ in Lemma \ref{proximal mapping for our problem} with $x\neq 0$, and $\psi$ in \eqref{capped function form for our problem}. Define $f^{\rm C-log}: \R_{+} \to \R$ as
	\begin{equation*}
		f^{\rm C-log}\left( u \right)=
		\begin{cases}
			\frac{1}{2}\left( u-\|x\| \right)^2 + \lambda \log \left(1+\frac{u}{\epsilon}\right), & {\rm if}\ 0 \leq u<\nu,\\
			\frac{1}{2}\left( u-\|x\| \right)^2 + \lambda C, & {\rm if}\  u \geq \nu.
		\end{cases}
	\end{equation*}
	Then the proximal mapping of $f$ at $x$ with the scaling parameter $\lambda>0$ is
	\begin{equation}\label{the form when xneq0}
		\prox_{\lambda f}(x)=
		\begin{cases}
			u_{1}^{*}(\|x\|) \frac{x}{\|x\|}, & {\rm if}\ f^{\rm C-log}\left( u_{1}^{*}\left(\|x\| \right) \right) \leq f^{\rm C-log}\left( u_{2}^{*}\left(\|x\| \right) \right),\\
			u_{2}^{*}(\|x\|) \frac{x}{\|x\|}, & {\rm otherwise},
		\end{cases}
	\end{equation}
	where
	$u_{1}^{*}(t)=\min \left\{ \left( \prox_{\lambda g}(t) \right)_{+},\nu\right\}$, $u_{2}^{*}(t)=\max \left\{ t,\nu \right\}$ for all $t\in \R$, and  $g(t)=\log \left(1+\frac{|t|}{\epsilon}\right)$ for all $t\in\R$.
\end{lemma}

 \begin{proof}
In view of the definition of proximal mapping, one has  
\[
\prox_{\lambda \psi^{\rm CapLog}}(\|x\|)= \argmin_{u\in\R_+} f^{\rm C-log}(u).
\]
For $0 \leq u <\nu$, we have $\prox_{\lambda \psi^{\rm CapLog}}(\|x\|)=\left(\prox_{\lambda g}(x)\right)_+$; for $u\geq \nu$, $\prox_{\lambda \psi^{\rm CapLog}}(\|x\|)=\|x\|$. Combining this with the analysis in \cite[Section A.2]{LiliPan2021}, the minimizer of $f^{\rm C-log}$ over $\left[0,\nu\right]$ is $u_1^{*}(x)=\min \left\{ \left(\prox_{\lambda g}(x)\right)_+, \nu\right\} $, and the minimizer over $\left[\nu,\infty \right)$ is $u_2^{*}(x)=\max \left\{ x, \nu\right\} $. This directly implies that \eqref{the form when xneq0} holds, completing the proof.
 \end{proof}

\begin{remark}
Consider problem \eqref{reproblem in experiment}. Let  
\begin{equation*}
	\Psi^{\rm CapLog}(x):=\sum_{i=1}^q \psi^{\rm CapLog}(x_{\mathcal{G}_i})=\sum_{i=1}^q
 	f(x_{\mathcal{G}_i}).
\end{equation*}
 Then it follows that:
$
 	\Psi^{\rm CapLog}(x)=\sum\limits_{i=1}^q
 	f(x_{\mathcal{G}_i}),
$
 with $f$ defined as in Lemma \ref{proximal mapping for our problem}. In view of the separability of $\Psi^{\rm CapLog}$ with respect to $i$, we follow the similar argument to both the proof of \cite[Theorem 3]{PST2022} and the discussion in \cite[Section 5.1]{PB2014} to derive:
\begin{equation*}
	{\rm prox}_{\lambda \Psi^{\rm CapLog}}(x)=
	{\rm prox}_{\lambda f}(x_{\mathcal{G}_1}) \times
		{\rm prox}_{\lambda f}(x_{\mathcal{G}_2}) \times \cdots \times
	{\rm prox}_{\lambda f}(x_{\mathcal{G}_q}),
\end{equation*}
where $\prox_{\lambda f}$ is established in  \eqref{proximalM_logP} and \eqref{the form when xneq0}.
\end{remark}

\section*{Declarations}
\begin{itemize}
	\item[$\bullet$] {Conflict of interest/Competing interests:} Not Applicable.
\end{itemize}


\begin{thebibliography}{99}
%
%
\bibitem{AbernethyBachEvgeniou2009}
Abernethy, J., Bach, F., Evgeniou, T., Vert, J.P.:
\newblock A new approach to collaborative filtering: operator estimation with spectral regularization.
\newblock {\em Journal of Machine Learning Research,} 10(3), 803-826 (2009)

\bibitem{BanerjeeMerugu2005}
Banerjee, A., Merugu, S., Dhillon, I.S., Ghosh, J.:
\newblock Clustering with Bregman divergences. Journal of machine learning research.
\newblock {\em Journal of Machine Learning Research,} 6, 1705-1749 (2005)

\bibitem{CandesRombergTao2006}
Cand\`es, E.J., Romberg, J., Tao, T.:
\newblock Robust uncertainty principles: exact signal reconstruction from highly incomplete frequency information.
\newblock {\em IEEE Transactions on Information Theory,} 52(2), 489-509 (2006)

\bibitem{ChenLuPong2016}
Chen, X., Lu, Z., Pong, T.K.:
\newblock Penalty methods for a class of non-Lipschitz optimization problems.
\newblock {\em SIAM Journal on Optimization,} 26(3), 1465-1492 (2016)

\bibitem{ChenSelesnick2014}
Chen, P.Y., Selesnick, I.W.:
\newblock Group-sparse signal denoising: non-convex regularization, convex optimization.
\newblock {\em IEEE Transactions on Signal Processing,} 62(13), 3464-3478 (2014)



\bibitem{EladAharon2006}
Elad, M., Aharon, M.:
\newblock Image denoising via sparse and redundant representations over learned dictionaries.
\newblock {\em IEEE Transactions on Image Processing,} 15(12), 3736-3745 (2006)

\bibitem{FanLi2001}
Fan, J., Li, R.:
\newblock Variable selection via nonconcave penalized likelihood and its oracle properties.
\newblock {\em Journal of the American Statistical Association,} 96(456), 1348-1360 (2001)

\bibitem{Fazel2013}
Fazel, M., Pong, T.K., Sun, D., Tseng,  P.:
\newblock Hankel matrix rank minimization with applications to system identification and realization.
\newblock {\em SIAM Journal on Matrix Analysis and Applications,} 34(3), 946-977 (2013)

\bibitem{HuLiMengQinYang2017}
Hu, Y., Li, C., Meng, K., Qin, J., Yang, X.:
\newblock Group sparse optimization via $\ell_{p,q}$ regularization.
\newblock {\em Journal of Machine Learning Research,} 18(30), 1-52 (2017)

\bibitem{HuberPJ1992}
Huber, P.J.:
\newblock Robust estimation of a location parameter.
\newblock {\em The Annals of Mathematical Statistics,} 35(1), 73-101 (1964)


\bibitem{HuangZhang2010}
Huang, J., Zhang, T.:
\newblock The benefit of group sparsity.
\newblock {\em Institute of Mathematical Statistics,} 38(4), 1978-2004 (2010)

\bibitem{HuangJBPMS2012}
Huang, J., Breheny, P., Ma, S.:
\newblock A selective review of group selection in high dimensional models.
\newblock {\em Statistical Science,} 27(4), 481-499 (2012)

\bibitem{HastieTFJTR2001}
Hastie, T., Tibshirani, R., Friedman, J.:
\newblock The Elements of Statistical Learning.
\newblock {\em Springer, New York} (2009)

\bibitem{JHCZ2024}
Jiang, L., Huang, Z., Chen, Y., Zhu, W.:
\newblock Iterative-weighted thresholding method for group-sparsity-constrained optimization with applications.
\newblock {\em IEEE Transactions on Neural Networks and Learning Systems,} 36(6), 11602-11616 (2025)

\bibitem{JiangDF2012}
Jiang, D.:
\newblock Concave selection in generalized linear models.
\newblock {\em Doctoral Dissertation, University of Iowa} (2012)

\bibitem{KorenBellVolinsky2009}
Koren, Y., Bell, R., Volinsky, C.:
\newblock Matrix factorization techniques for recommender systems.
\newblock {\em Computer,} 42(8), 30-37 (2009)

\bibitem{LSM2020}
Li, J., So, A.M.-C., Ma, W.K.:
\newblock Understanding notions of stationarity in nonsmooth optimization.
\newblock {\em IEEE Signal Processing Magazine,} 37(5), 18-31 (2020)

\bibitem{LiLuDongTao2018}
Li, X., Lu, Q., Dong, Y., Tao, D.:
\newblock Robust subspace clustering by Cauchy loss function.
\newblock {\em IEEE Transactions on Neural Networks and Learning Systems,} 30(7), 2067-2078 (2018)

\bibitem{MostellerTukey1977}
Mosteller, F., Tukey, J.W.:
\newblock Data Analysis and Regression: A Second Course in Statistics.
\newblock {\em Addison Wesley, Sydney} (1977)

\bibitem{PengChen2020}
Peng, D., Chen, X.:
\newblock Computation of second-order directional stationary points for group sparse optimization.
\newblock {\em Optimization Methods and Software,} 35(2), 348-376 (2019)

\bibitem{LiliPan2021}
Pan, L., Chen, X.:
\newblock Group sparse optimization for images recovery using capped folded concave functions.
\newblock {\em SIAM Journal on Imaging Sciences,} 14(1), 1-25 (2021)

\bibitem{PST2022}
Prater-Bennette, A., Shen, L., Tripp, E.E.:
\newblock The proximity operator of the log-sum penalty.
\newblock {\em Journal of Scientific Computing,} 93(3), 67 (2022)

 \bibitem{PB2014}
 Parikh, N., Boyd, S.:
 \newblock Proximal algorithms.
 \newblock {\em Foundations and Trends in Optimization,} 1(3), 127-239 (2014)

\bibitem{QinHuYaoLeung2021}
 Qin, J., Hu, Y., Yao, J.C., Leung, R.W.T., Zhou, Y., Qin, Y., Wang, J.:
 \newblock Cell fate conversion prediction by group sparse optimization method utilizing single-cell and bulk OMICs data.
 \newblock {\em Briefings in Bioinformatics,} 22(6), bbab311 (2021)

\bibitem{variationalanalysis}
Rockafellar, R.T., Wets, R.J.:
\newblock  Variational Analysis.
\newblock {\em Springer Berlin, Heidelberg} (1998)


\bibitem{SSQTKP23}
Sun, S., Pong, T.K.:
\newblock Doubly iteratively reweighted algorithm for constrained compressed sensing models.
\newblock {\em Computational Optimization and Applications,} 85(2), 583-619 (2023)

\bibitem{Schapire1990} 
Schapire, R.E.:
\newblock The strength of weak learnability.
\newblock {\em Machine learning,} 5(2), 197-227 (1990)

\bibitem{SolodovSvaiter2000}
Solodov, M.V., Svaiter, B.F.:
\newblock An inexact hybrid generalized proximal point algorithm and some new results on the theory of Bregman functions.
\newblock {\em Mathematics of Operations Research,} 25(2), 214-230 (2000)

\bibitem{SchusterTrettnerKobbelt2020}
Schuster, K., Trettner, P., Kobbelt, L.:
\newblock High-performance image filters via sparse approximations.
\newblock {\em Association for Computing Machinery,} 3(2),
1-19 (2020)

\bibitem{SimonFriedman2013}
Simon, N., Friedman, J., Hastie, T., Tibshirani, R.:
\newblock A sparse-group lasso.
\newblock {\em Journal of Computational and Graphical Statistics,} 22(2),
231-245 (2013)

\bibitem{TuZhangGao2020}
Tu, K., Zhang, H., Gao, H., Feng, J.:
\newblock A hybrid Bregman alternating direction method of multipliers for the linearly constrained difference-of-convex problems.
\newblock {\em Journal of Global Optimization,} 76(4), 665-693 (2020)

\bibitem{TangNehorai2014}
Tang, G., Nehorai, A.:
\newblock Computable performance bounds on sparse recovery.
\newblock {\em IEEE Transactions on Signal Processing,} 63(1),
132-141 (2015)

\bibitem{TanEldarNehorai2014}
Tan, Z., Eldar, Y.C., Nehorai, A.:
\newblock Direction of arrival estimation using co-prime arrays: a super resolution viewpoint.
\newblock {\em IEEE Transactions on Signal Processing,} 62(21),
5565-5576 (2014)

\bibitem{WangLChenG2007}
Wang, L., Chen, G., Li, H.:
\newblock Group SCAD regression analysis for microarray time course gene expression data.
\newblock {\em Bioinformatics,} 23(12),
1486-1494 (2007)

\bibitem{YuanLin2006}
Yuan, M., Lin, Y.:
\newblock Model selection and estimation in regression with grouped variables.
\newblock {\em Journal of the Royal Statistical Society Series B: Statistical Methodology,} 68(1),
49-67 (2005)

\bibitem{ZhangC2010}
Zhang, C.H.:
\newblock Nearly unbiased variable selection under minimax concave penalty.
\newblock {\em The Annals of Statistics,} 38(2),  894-942 (2010)

\bibitem{ZhangYongle2023}
Zhang, Y., Li, G., Pong, T.K., Xu, S.:
\newblock Retraction-based first-order feasible methods for difference-of-convex programs with smooth inequality and simple geometric constraints.
\newblock {\em Advances in Computational Mathematics,} 49(1), 8 (2023)

\bibitem{ZhaWenYuan2021}
Zha, Z., Wen, B., Yuan, X., Zhou, J., Zhu, C.:
\newblock Image restoration via reconciliation of group sparsity and low-rank models.
\newblock {\em IEEE Transactions on Image Processing,} 30, 5223-5238 (2021)

\bibitem{ZhangYangChenYang2018}
Zhang, F., Yang, Z., Chen, Y., Yang, J., Yang, G.:
\newblock Matrix completion via capped nuclear norm.
\newblock {\em IET Image Processing,} 12(6), 959-966 (2018)

\bibitem{ZhangPeng2022}
Zhang, X., Peng, D.:
\newblock Solving constrained nonsmooth group sparse optimization via group Capped-$\ell_1$ relaxation and group smoothing proximal gradient algorithm.
\newblock {\em Computational Optimization and Applications,} 83(3), 801-844 (2022)
\end{thebibliography}


\end{document}